\documentclass[12pt]{article}

\usepackage[utf8]{inputenc}
\usepackage{epsfig}
\usepackage{graphicx}
\usepackage{amsfonts}
\usepackage{amsmath}
\usepackage{enumerate}
\usepackage{lineno}
\usepackage{setspace}
\usepackage{amssymb}
\usepackage{hyperref}
\usepackage{geometry}

\onehalfspace

\newtheorem{theorem}{Theorem}[section]
\newtheorem{corollary}[theorem]{Corollary}
\newtheorem{proposition}[theorem]{Proposition}
\newenvironment{proof}{\par \noindent \textit{Proof}.}{\hfill $\Box$\newline}

\newcommand{\valI}[2][G^2]{val_i(#2,#1)}
\newcommand{\valE}[2][G^2]{val_e(#2,#1)}
\newcommand{\valT}[3][G^1]{val(#1,#2,#3)}
\newcommand{\val}[1]{val(G^1,G^2,#1)}

\newcommand{\ctwo}{2 \lfloor \frac{m}{2} \rfloor}
\newcommand{\cthreeC}{m + \max\{2, \lfloor \frac{m-2}{2} \rfloor\}}
\newcommand{\cthreeCC}{ 3 \lceil \frac{m}{2} \rceil}
\newcommand{\cthreeCP}{ 3 \lfloor \frac{m}{2} \rfloor}
\newcommand{\cthreeP}{m + \lfloor \frac{m-2}{2} \rfloor}
\newcommand{\cthreeCPtwo}{\max\{2, \lfloor \frac{m-2}{2} \rfloor\}}
\newcommand{\mtwoovertwo}{\lfloor \frac{m-2}{2} \rfloor}
\newcommand{\cfour}{2m - 1}
\newcommand{\cfiveC}{\max\{m+4, 2m-1\}}
\newcommand{\cfiveP}{2m-1}
\newcommand{\csixC}{2m + 4}
\newcommand{\csixP}{2m + 2}
\newcommand{\csevenC}{3m + 1}
\newcommand{\csevenP}{3m}

\begin{document}

\title{Global Defensive Alliances in the \\ Lexicographic Product of Paths and Cycles}

\author{Rommel M. Barbosa$^1$\thanks{E-mail: \textit{rommel@inf.ufg.br}} \and
Mitre C. Dourado$^2$\thanks{Partially suported by Conselho Nacional de Desenvolvimento Científico e Tecnológico (CNPq, Brazil). E-mail: \textit{mitre@dcc.ufrj.br}} \and
Leila R. S. da Silva$^3$\thanks {Partially suported by Fundação de Amparo à Pesquisa do Estado de Goiás (FAPEG, Brazil). E-mail: \textit{leila.roling@ifgoiano.edu.br}}}

\maketitle

\begin{center}
{\small
$^1$Instituto de Informática, Universidade Federal de Goiás, Goiânia, Brazil \\

$^2$Instituto de Matemática, Universidade Federal do Rio de Janeiro, Rio de Janeiro, Brazil \\

$^3$Instituto Federal Goiano, Campus Mourinhos, Brazil}
\end{center}

\begin{abstract}
A set $S$ of vertices of graph $G$ is a \textit{defensive alliance} of $G$ if for every $v \in S$, it holds $|N[v] \cap S| \geq |N[v]-S|$. An alliance $S$ is called $global$ if it is also a dominating set. In this paper, we determine the exact values of the global defensive alliance number of lexicographic products of path and cycles.
\end{abstract} 

\textit{Mathematical Sciences Classification: 05C76}


\section{Introduction}

We consider only finite, simple, and undirected graphs. Given a graph $G = (V,E)$, {\em open neighborhood} and the {\em closed neighborhood} of a $v \in V$ are denoted by $N(v), N[v]$, respectively. 
Given a set $S \subseteq V$, the subgraph of $G$ induced by $S$ is denoted by $G[S]$. If $v \in S$ and $|N[v] \cap S| \geq |N[v] - S|$, then $v$ is said to be {\em defended in $S$}. We say that $S$ is a {\em defensive alliance} if all vertices of $S$ are defended. Note that if $v$ is defended in $S$, then $|S \cap N(v)| \geq \lfloor \frac{d(v)}{2} \rfloor$.
The set $S$ is a {\em dominating set of $G$} if every vertex of $G$ belongs to $S$ or has a neighbor is $S$. A defensive alliance is {\em global} (GDA) if it is also a dominating set of the graph. The minimum cardinality of a global defensive alliance of $G$ is its {\em global defensive alliance number} and is denoted by $\gamma_a(G)$.

The {\em lexicographic product} of graphs $G^1 = (V_1,E_1)$ and $G^2 = (V_2,E_2)$ is the graph $G = (V,E) = G^1 \circ G^2$ such that $V = V_1 \times V_2$ and $E = \{(u_1,u_2)(v_1,v_2) : (u_1v_1 \in E_1) \mbox{ or } (u_1 = v_1$ \mbox{ and } $u_2v_2 \in E_2)\}$.
Given a graph $F = G^1 \circ G^2$ where the orders of $G^1$ and $G^2$ are $n$ and $m$, respectively, it is clear that $F$ contains $n$ disjoint copies of $G^2$, which will be denoted by $G^2_1, G^2_2, \ldots, G^2_n$. Furthermore, for a set $S \subseteq V(G)$, we will denote by $S_i$ the set $S \cap V(G^2_i)$ for $i \in \{1, \ldots, n\}$ and $s_i = |S_i|$. 

In this work, we present formulas that allow one determining the global defensive alliance number of a graph $F = G^1 \circ G^2$ where $G^1$ and $G^2$ are cycles or paths within a constant number of operations. 
In Section~\ref{sec:charac}, we present a general characterization of $\gamma_a(G^1 \circ G^2)$ for $G^1 \in \{P_n,C_n\}$, $n \geq 3$, and any graph $G^2 \not\simeq K_m$. Such characterization will be useful for the proposed solution presented in next sections.
Section~\ref{sec:first} contains useful properties of minimum global defensive alliances of the lexicographic product of paths and cycles.
In Section~\ref{sec:nsmall}, we present the formulas for $n \leq 7$ while the solution for $n \geq 8$ is given in Section~\ref{sec:nbig}. In Section~\ref{sec:imp}, we explore the homogenous bevavior of $\gamma_a({G^1 \circ G^2})$, when the orders of $G_1$ and $G_2$ change, for obtaining more structural results. The conclusions are in Section~\ref{sec:conclusion}. We finish this section presenting related works.

The definition of alliances in graphs first appeard in~\cite{HHH2003}. Since then many variatons appeared.
The most extensively studied are defensive alliances \cite{HHH2003,Hed2004,Rod2009,Sig2011,Yer2010}, offensive alliances \cite{Fav2004,Rod2006,Sig2009} and powerful or dual alliances \cite{Bri2004, Bri2009, Yer2012}. A more generalized concept of alliance is represented by \textit{k-alliances} \cite{Ber2010,Sha2002, Sha2003, Sha2006,Sig2011}, and  Dourado et al. presented a new definition of alliances, namely, \textit{$(f,g)$-alliances}~\cite{Dou2011}, that generalizes previous concepts. In \cite{Yer2013c}, Yero and Rodr\'{i}guez-Vel\'{a}zquez published a summary of the major results obtained concerning defensive alliances up through $2013$.

Since the decision problems of computing the minimum cardinality of these concepts for general graphs are NP-complete~\cite{Cam2006, Fer2008,Rod2009b}, several studies of alliances in graphs have been developed in graph classes and product of graphs; these advances are described in detail in \cite{Yer2013c,Fer2014}.

Haynes, Hedetniemi and Henning \cite{HHH2003} determined the cardinality of the minimum set that can constitute a global defensive alliance for several classes of graphs and presented some limits on the minimum GDA in cubic, bipartite graphs and trees. 

The initial studies of defensive alliances in Cartesian products were done by Brigham, Dutton and Hedetniemi in \cite{Bri2004}, and several parameters were also presented in \cite{Ber2010,Sig2011,Yer2011} for Cartesian products of graphs for k-alliances.
Following this trend, there is also the work by Chang \textit{et al.} in 2012 \cite{Cha2012}, which presented some upper bounds for Cartesian products between paths and cycles. In 2013, Yero and Rodr\'{i}guez-Vel\'{a}zquez \cite{Yer2013} obtained closed formulas for the GDA number for several classes of Cartesian products of graphs.

\section{Characterization of $\gamma_a(G^1 \circ G^2)$ for $G^1 \in \{C_n,P_n\}$} \label{sec:charac}

In this section, we present a characterization of $\gamma_a(G^1 \circ G^2)$ for $G^1 \in \{P_n,C_n\}$, $n \geq 3$, and any graph $G^2 \not\simeq K_m$. Let $S$ be a GDA of $F = G^1 \circ G^2$. The {\em spectrum of $S$ in $F$}, $spe(S,F)$, is a sequence obtained in the following way. If $G^2 = C_m$ and there is $S_i = 0$, we assume that $S_n = 0$: if $S_i \neq 0$ for $2 \leq i \leq n-1$, then $spe(S,F) = (n)$; otherwise let $i \geq 3$ be the minimum number such that $S_i = 0$. If $S_{i+1} = 0$, then $k_1 = i$; otherwise $k_1 = i-1$. In both cases, $spe(S,F) = (k_1)spe(S',F')$ where $S' = S \cap V(F')$ and $F' = F - (V(G^2_1) \cup  \ldots \cup V(G^2_{k_1}))$. When there is no doubt about which is the graph $F$, we can use $spe(S)$ to represent the spectrum of $S$ in $F$.

We say that a sequence $w = (k_1, \ldots, k_t)$ is {\em feasible} for $G^1 \circ G^2$ for $G^1 \in \{C_n,P_n\}$ and $G^2 \not\simeq K_m$ if there is a GDA $S$ of $G^1 \circ G^2$ whose spectrum is $w$. We denote $max(w) = \underset{i \in t}{max}\{k_i\}$ and say that $w$ is an $n$-sequence if $\underset{1 \leq i \leq t}{k_i} = n$. Observe that we can see $w$ as a $t$-partition $(V_1, \ldots, V_t)$ of $V(F)$ where each part is associated with an element $k_i$ and $F[V_i] \simeq P_{k_i} \circ G^2$. We call each such subgraph by a {\em section of $F$}. If $G^1 \simeq P_n$ and $F_i$ is a section for $i \in \{1,t\}$, then we say that $F_i$ is an {\em external section}; otherwise it is an {\em internal section}.

Given elements $k_i$ and $k_j$ of a sequence $w = (k_1, \ldots, k_t)$ for $j > i$, the sequence formed by the elements that are between $k_i$ and $k_j$ in $w$ will be denoted by $w_{i+1,j-1} = (k_{i+1}, \ldots, k_{j-1})$, the sequence formed by the elements that preceed $k_i$ by $w_{1,i-1} = (k_1, \ldots, k_{i-1})$, and the sequence formed by the elements that succeed $k_j$ by $w_{j+1,t} = (k_{j+1}, \ldots, k_t)$. The concatenation of sequences $w$ and $w'$ will be denoted by $ww'$. If all elements of $w$ are equal, then we can write $w = ([t]k_1)$. This definition allows one to write $w = ([t_1]k_1, \ldots, [t_p]k_p)$, which means that, for $1 \leq i \leq p$, there are $t_i$ consecutive occurrences of $k_i$ and $\underset{1 \leq i \leq p}{t_ik_i} = n$. The feasible sequences are characterized in the following result.

\begin{proposition}
For $G^1 \in \{P_n,C_n\}$ and $G^2 \not\simeq K_m$, a sequence $w = (k_1, \ldots, k_t )$ is feasible for $G^1 \circ G^2$ if and only if  $k_1 \geq 2$, $k_i \geq 3$ for $i \in \{2, \ldots, t\}$, and $\underset{1 \leq i \leq t}{\sum} k_i = n$.
\end{proposition}

\begin{proof}
Let $S$ be a GDA of $F$ such that $spe(S) = w$. By the construction of a spectrum, it is clear that $\underset{1 \leq i \leq t}{\sum} k_i = n$. Since every vertex of $S$ has a neighbour outside the copy of $G^2$ that it belongs, every $k_i \geq 2$. Since the definition of spectrum guarantees that for every $k_i$, for $i \geq 2$, the first copy of the section associated with it has no vertex of $S$, $k_i \geq 3$ for $i \geq 3$. 

Conversely, let $F_i = F[V(G^2_j) \cup \ldots \cup V(G^2_{j+k_i-1})]$ be the section associated with $k_i$.
If $k_i \geq 4$, add $V(G^2_{j+1}) \cup \ldots \cup V(G^2_{j+k_i-2})$ to $S$.
If $k_i = 3$, add $V(G^2_{j+1}) \cup V(G^2_{j+2})$ to $S$.
If $k_1 = 2$, add $V(G^2_1) \cup V(G^2_2)$ to $S$. It is clear $S$ is a dominating set,  every vertex of $S$ is defended, and that $spe(S) = w$, then $w$ is feasible.
\end{proof}

For the characterization of minimum GDA in terms of feasible sequences, we need some definitions. Given a positive integer $k$, we define 

\begin{itemize}

\item for $k \geq 4$, $\valI{k}$ as the cardinality of a minimum GDA $S$ of $P_k \circ G^2$ such that $s_1 = s_k = 0$;

\item for $k \in \{2,3\}$, $\valI{k} = \valI{4}$;

\item for $k \geq 3$, $\valE{k}$ as the minimum GDA $S$ of $P_k \circ G^2$ such that $s_1 = 0$;

\item $\valE{2} = \valE{3}$.

\end{itemize}

For a sequence $w = (k_1, \ldots, k_t)$, we define

\begin{itemize}
\item $\valT[P_n]{G^2}{w} = \valE{k_1} + \underset{2 \leq i \leq t-1}{\sum} \valI{k_i} + \valE{k_t}$;

\item $\valT[C_n]{G^2}{w} = \underset{1 \leq i \leq t}{\sum} \valI{k_i}$.
\end{itemize}

\begin{proposition} \label{pro:valGDA}
If $S$ is a GDA of $F = G^1 \circ G^2$ for $G^1 \in \{C_n, P_n\}$ and $G^2 \not\simeq K_m$, then $|S| \geq \val{spe(S)}$ for $G \in \{P,C\}$. Furthermore, there is GDA of $F$ of size $\val{spe(S)}$.
\end{proposition}

\begin{proof}
Write $w = (k_1, \ldots, k_t)$ and let $F_i = F[V(G^2_j) \cup \ldots \cup V(G^2_{j+k_i-1})]$ be the section associated with $k_i$ and $S' = S \cap F_i$. First, consider $G = C$. For $k_i \geq 4$, since $s_{j+k_i} = 0$, it holds $|S'| \geq \valI{k_i}$. For $k_i = 3$, since $s_{j+3} = 0$, it holds $|S'| \geq \valI{4}$.
Finally for $k_1 = 2$, since $s_{3} = s_n = 0$, it holds $|S'| \geq \valI{4}$. Then, the result holds for $G = C$. Now consider $G = P$. For $k_i \geq 3$, since $s_j = 0$ or $s_{j+k_i} = 0$, it holds $|S'| \geq \valE{k_i}$. Finally for $k_1 = 2$, since $s_{3} = 0$, it holds $|S'| \geq \valE{3}$, completing the proof.
\end{proof}

\begin{corollary} \label{cor:dc}
Let $F = G^1 \circ G^2$ for $G^1 \in \{C_n, P_n\}$ and $G^2 \not\simeq K_m$. Then $\gamma_a(F) = \min\{val(w)\}$ where $w$ is a feasible sequence for $G^2$.
\end{corollary}

As a byproduct, we have that, for $G^1 \in \{C_n,P_n\}$ and $G^2 \not\simeq K_m$, if the number of feasible sequences $w$ that can reach the minimum GDA of $G^1 \circ G^2$ is bounded by a polynomial on $n$ and $m$, one can determine them efficiently, and the values $\valI{k}$ and $\valE{k}$ are known for every $k \leq max(w)$, one can find $\gamma_a(G^1 \circ G^2)$ efficiently. We show in next sections that this does hold for $G^2 \in \{P_m, C_m\}$. The last result of this section deals with the external sections.

\begin{corollary} \label{cor:external23}
If $w$ is a feasible sequence of $F = P_n \circ G^2$ for $G^2 \not\simeq K_m$ such that $\valT[P_n]{G^2}{w} = \gamma_a(F)$, then the following hold:

\begin{itemize}
	\item if $3$ occurs in $w$, we can assume that $k_t = 3$;
	\item if $k_1 \neq 2$ and there are two occurrences of $3$ in $w$, we can assume that $k_1 = k_t = 3$.
\end{itemize}
\end{corollary}

\section{Properties of global defensive alliances} \label{sec:first}

In this section, we present some bounds and properties of GDAs that will be useful in the remaining sections.

\begin{proposition} \label{pro:basic}
Let $S$ be a GDA of $G^1 \circ G^2$ for $G^1 \in \{P_n, C_n\}$, $G^2 \in \{P_m, C_m\}$, $n \geq 3, m \geq 3$, and $i$ be an integer such that $2 \leq i \leq n-1$. Then, the following setences hold

\begin{enumerate} [$(i)$]
	\item \label{ite:m2} If $G^2 \simeq \{C_m,P_3\}$, then $s_{i-1} + s_i + s_{i+1} \geq m+2$;
	
	\item If $G^2 \simeq P_m$ for $m \geq 4$, then $s_{i-1} + s_i + s_{i+1} \geq m+1$;
	
	\item If $s_i \geq 1$, then $s_{i-1} + s_{i+1} \geq m-1$; and  \label{ite:basic}
		
	\item If $1 \leq s_i < m$, then $s_{i-1} + s_{i+1} \geq m$.
\end{enumerate}
\end{proposition}

\begin{proof}
Let $v \in S_i$ and $d$ be the number of neighbors of $v$ in $S_i$.

\bigskip
\noindent $(i)$  Since $d(v) = 2m+2$, $S$ must contain at least $m+1$ neighbors of $v$. This means that $s_{i-1} + s_i + s_{i+1} \geq m+2$ because $N[v] \subseteq V(G^2_{i-1}) \cup V(G^2_i) \cup V(G^2_{i+1})$.

\bigskip 
\noindent $(ii)$ Since $d(v) = 2m+1$, $S$ must contain at least $m$ neighbors of $v$. This means that $s_{i-1} + s_i + s_{i+1} \geq m+1$ because $N[v] \subseteq V(G^2_{i-1}) \cup V(G^2_i) \cup V(G^2_{i+1})$.

\bigskip
\noindent $(iii)$ Consequence of $(i), (ii),$ and $1 \leq |N(v) \cap V(G^2_i)| \leq 2$.

\bigskip
\noindent $(iv)$ Consequence of $(iii)$ and the fact that $v$ can be chosen as a vertex having a neighbor in $V(G^2_i) \setminus S$.
\end{proof}

\begin{proposition} \label{pro:lower}
Let $S$ be a GDA of $G^1 \circ G^2$ for $G^1 \in \{P_n, C_n\}$, $G^2 \in \{P_m, C_m\}$, $n \equiv r \mod 4$, and $m \geq 3$ such that $s_i \geq 1$ for every $2 \leq i \leq n-1$. Then the following hold.
	
	\begin{enumerate} [$(i)$]
	\item \label{ite:r0} If $r = 0$, then $|S| \geq (2m-1) \frac{n}{4}$		
		
	\item \label{ite:r123} If $r \in \{1,2,3\}$ and $n \geq 8$, then
	$|S| \geq (2m-1) \lfloor\frac{n}{4} \rfloor  + t$, where

	$t = 
	\begin{cases}
	m + 1 & \mbox{, if } r = 3 \mbox{ and } G^2 \simeq P_m\\
	m + 2 & \mbox{, if } r = 3 \mbox{ and } G^2 \simeq C_m \\
	r & \mbox{, if } r \in \{1,2\}
	\end{cases}$

	\item \label{ite:m3} If $n \geq 6$ and $m = 3$, then $|S| \geq 6 \lfloor\frac{n}{4} \rfloor  + t$, where	
	$t =
	\begin{cases}
	r + 2 & \mbox{, if } r \in \{1,2,3\} \\
	0 & \mbox{, if } r = 0	
	\end{cases}$

	\item \label{ite:n9m4} If $n \geq 9$ and $m = 4$, then $|S| \geq 2n$ for $G^2 \simeq C_m$ and $|S| \geq 2n-2$ for $G^2 \simeq P_m$ 
	
	\end{enumerate}
\end{proposition}

\begin{proof}
$(i)$ By Proposition~\ref{pro:basic}, $s_1 + s_3 \geq m-1$ and $s_2 + s_4 \geq m-1$. If $s_1 + s_3 = s_2 + s_4 = m - 1$, then $S$ is not a GDA because some vertex of $S_2$ is not defended in $S$. Then $s_1 + s_2 + s_3 + s_4 \geq 2m-1$. In fact, we can conclude that $s_i + s_{i+1} + s_{i+2} + s_{i+3} \geq 2m-1$ for every $i \in \{1, \ldots, n - 3\}$.

\bigskip
\noindent $(ii)$
Since $n \geq 8$, $n - 4 - r = 4k$ for some positive integer $k$.
If suffices to show that for $T = (V(G^2_5) \cup \ldots \cup V(G^2_{5+r-1})) \cap S$ it holds $|T| \geq t$.
If $r \leq 2$, then $|T| \geq r$ because $s_i \geq 1$ for every $i \in \{2, \ldots, n-1\}$. If $r = 3$, then $|T| \geq m + 1$ if $G^2 \simeq P_m$ and $|T| \geq m + 2$ if $G^2 \simeq C_m$ due Proposition~\ref{pro:basic}.

\bigskip
\noindent $(iii)$ By Proposition~\ref{pro:basic}~$(\ref{ite:m2})$, it holds $s_i + s_{i+1} + s_{i+2} + s_{i+3} \geq 6$ for every $i \in \{1, \ldots, n - 3\}$. Then, the result is clear for $r = 0$. Case $r = 1$ is consequence of the fact that $s_1 + \ldots + s_9 \geq 15$, case $r = 2$ because $s_1 + \ldots + s_6 \geq 10$, and case $r = 3$ because $s_1 + \ldots + s_7 \geq 11$.

\bigskip
\noindent $(iv)$ It suffices to prove for $G^2 \simeq P_m$. We prove that if $s_i = 1$, then $s_{i+1} \geq 3$ or $s_{i+1} + s_{i+2} \geq 5$ or $s_{i+1} + s_{i+2} + s_{i+3} \geq 7$. If $s_{i+1} \geq 3$, we are done. If $s_{i+1} = 1$, then $s_{i+2} \geq 3$. If then $s_{i+2} = 3$, then there is a vertex of degree 2 in $S_{i+2}$ having a neighbor in $V(G^2_2) \setminus S$, therefore $s_{i+3} \geq 3$. Then consider $s_{i+3} = 4$ and $s_{i+1} + s_{i+2} \geq 5$. Then consider $s_{i+1} = 2$. This means that there is a vertex of degree 2 in $S_{i+1}$ having a neighbor in $V(G^2_1) \setminus S$, therefore $s_{i+2} \geq 3$ and $s_{i+1} + s_{i+2} \geq 5$.

Now, it remains to recall that $s_1 + s_2 + s_3 \geq 6$ and $s_{n-2} + s_{n-1} + s_n \geq 6$ for $G^2 \simeq C_4$ and $s_1 + s_2 + s_3 \geq 5$ and $s_{n-2} + s_{n-1} + s_n \geq 5$ for $G^2 \simeq P_4$.
\end{proof}

\begin{proposition} \label{pro:spanning}
If $G$ is a spanning subgraph of $G'$ and $S$ is a minimum GDA of $G$ such that no vertex of $S$ is incident to any edge of $E(G') \setminus E(G)$, then $S$ is also a minimum GDA of $G'$.
\end{proposition}

\begin{proof}
	Consequence of the fact that the neighborhood  of each vertex of $S$ is the same in $G$ and in $G'$.
\end{proof}

\section{Determining $\gamma_a$ for paths and cycles} \label{sec:Pn}

For $n \geq 3$ and $m \geq 2$, we show in this section that $\gamma_a(G_1,G_2)$ for $G^1,G^2 \in \{C_n,P_m\}$ is the minimum among at most four values. Since these values are easily evaluated, one can determining $\gamma_a(G^1,G^2)$ within a constant number of operations. We consider first the case where $G^1$ has order at most 7.

\subsection{Case $n \leq 7$} \label{sec:nsmall}

Let $F = G^1 \circ G^2$, for $G^1 \simeq P_n$, $G^2 \simeq C_m$, $n \in \{2, \ldots, 7\}$, and $m \geq 3$. We define $X_{n,m} \subseteq V(F)$ as follows:

\begin{itemize}
\item $X_{2,3} = V(G^2_1)$. For $m \geq 4$, define $X_{2,m} = T_1 \cup T_2$, where $T_1$ and $T_2$ are the vertex sets of paths of order $\lfloor\frac{m}{2}\rfloor$ of $G^2_1$ and $G^2_2$, respectively.

\item $X_{3,m} = V(G^2_3) \cup T_2$, where $T_2$ is the vertex set of a path of order $x$ of $G^2_2$ where $x = \max \{2, \lfloor \frac{m-2}{2}\rfloor\}$.

\item $X_{4,m} = V(G^2_3) \cup V(G^2_2) - \{u\}$, for some vertex $u \in V(G^2_2)$.

\item $X_{5,m} = V(G^2_3) \cup T_2 \cup T_4$ where $T_2$ contains two adjacent vertices of $G^2_2$ and $T_4$ contains the vertices of a path of $G^2_4$ of size $\max\{2, m - 3\}$.

\item $X_{6,3} = V(G^2_1) \cup V(G^2_5) \cup T_4$, where $T_4$ is a pair of vertices; for $m \geq 4$, define $X_{6,m} = V(G^2_3) \cup V(G^2_4) \cup T_2 \cup T_5$, where $T_2$ and $T_5$ are two adjacent vertices of $G^2_2$ and $G^2_5$, respectively.

\item $X_{7,m} = V(G^2_3) \cup V(G^2_5) \cup T_2 \cup T_4 \cup T_6$, where $T_2$ and $T_6$ are two adjacent vertices of $G^2_2$ and $G^2_6$, respectively, and $T_4$ is the vertex set of a path of order $m-3$ of $G^2_4$.

\end{itemize}

\begin{table}[h]
	\centering
	\begin{tabular}{|l|l|l|}
		\hline
		$n$ & $P_n \circ C_3$ & $P_n \circ C_m, m \geq 4$  \\
		\hline
		$2$ & $3$ & $\ctwo $ \\
		
		$3$ & $5$ & $\cthreeC$  \\
		
		$4$ & $5$ & $\cfour$  \\
		
		$5$ & $7$ & $\cfiveC$ \\
		
		$6$ & $8$ & $\csixC$ \\
		
		$7$ & $10$ & $\csevenC$ \\
		\hline
	\end{tabular}
	\caption{$P_n \circ C_m, 2 \leq n \leq 7$ and $m \geq 3$.}
	\label{tab:P2-7Cm}
\end{table}

\begin{table}[h]
	\centering
	\begin{tabular}{|l|l|l|}
		\hline
		$n$ & $C_n \circ C_3$ & $C_n \circ C_m, m \geq 4$  \\
		\hline
		$3$ & $5$ & $\cthreeCC$ \\
		
		$4$ & $5$ & $\cfour$  \\
		
		$5$ & $7$ & $\cfiveC$ \\
		
		$6$ & $10$ & $\csixC$ \\
		
		$7$ & $10$ & $\csevenC$ \\
		\hline
	\end{tabular}
	\caption{$C_n \circ C_m, 3 \leq n \leq 7$ and $m \geq 3$.}
	\label{tab:C2-7Cm}
\end{table}

Let $F = G^1 \circ G^2$, for $G^1 \simeq P_n$, $G^2 \simeq P_m$, $n \in \{2, \ldots, 7\}$, and $m \geq 3$. We define $Y_{n,m} \subseteq V(F)$ as follows:

\begin{itemize}
	
\item $Y_{2,3} = V(G^2_1) \setminus \{v_1\} \cup \{v_2\}$. For $m \geq 4$, define $X_{2,m} = T_1 \cup T_2$, where $T_1$ and $T_2$ are the vertex sets of paths of order $\lfloor\frac{m}{2}\rfloor$ of $G^2_1$ and $G^2_2$, respectively.

\item For $m \in \{3,4\}$, $Y_{3,m} = V(G^2_3) \cup  \{v_2\}$ for $v_2 \in V(G^2_2)$; $Y_{3,5} = V(G^2_3) \setminus \{u\} \cup T_2$ where $T_2$ contains two adjacent vertices of $G^2_2$; for $m \geq 6$, $Y_{3,m} = V(G^2_3) \cup T_2$, where $T_2$ is the vertex set of a path of $G^2_2$ of order $\lfloor \frac{m-2}{2}\rfloor$.

\item $Y_{4,m} = X_{4,m}$.

\item $Y_{5,3} = V(G^2_3) \cup \{v_1\} \cup \{v_4\}$; 
$Y_{5,4} = V(G^2_3) \cup \{v_2\} \cup \{v_4,v'_4\}$;
for $m \geq 5$, $Y_{5,m} = X_{5,m}$.

\item $Y_{6,3} = V(G^2_3) \cup V(G^2_4) \cup \{v_2\} \cup \{v_5\}$, where $d(v_2) = d(v_5) = 7$; for $m \geq 4$, define $X_{6,m} = V(G^2_3) \cup V(G^2_4) \cup \{v_2\} \cup \{v_5\}$.

\item $Y_{7,m} = V(G^2_3) \cup V(G^2_5) \cup \{v_2\} \cup T_4 \cup \{v_6\}$, where $T_4$ is the vertex set of a path of order $m-2$ of $G^2_4$.
	
\end{itemize}

\begin{table}[h]
	\centering
	\begin{tabular}{|l|l|l|}
		\hline
		$n$ & $P_n \circ P_3$ & $P_n \circ P_m, m \geq 4$  \\
		\hline
		$2$ & $3$ & $\ctwo$ \\
		
		$3$ & $4$ & $m + \mtwoovertwo$ \\
		
		$4$ & $5$ & $\cfour$ \\
		
		$5$ & $5$ & $\cfiveP$ \\
		
		$6$ & $8$ & $\csixP$ \\
		
		$7$ & $9$ & $\csevenP$ \\
		\hline
	\end{tabular}
	\caption{$P_n \circ P_m, 2 \leq n \leq 7$ and $m \geq 3$.}
	\label{tab:P2-7Pm}	
\end{table}

\begin{table}[h]
	\centering
	\begin{tabular}{|l|l|l|}
		\hline
		$n$ & $C_n \circ P_3$ & $C_n \circ P_m, m \geq 4$  \\
		\hline
		$3$ & $5$ & $\cthreeCP$ \\
		
		$4$ & $5$ & $\cfour$ \\
		
		$5$ & $5$ & $\cfiveP$ \\
		
		$6$ & $8$ & $\csixP$ \\
		
		$7$ & $9$ & $\csevenP$ \\
		\hline
	\end{tabular}
	\caption{$C_n \circ P_m, 3 \leq n \leq 7$ and $m \geq 3$.}
	\label{tab:C2-7Pm}	
\end{table}

Denote $x_{i,m} = |X_{i,m}|$ and $y_{i,m} = |Y_{i,m}|$.

\begin{proposition} \label{pro:p2-7G2}
For $n \in \{2, \ldots, 7\}$, $\gamma_a(P_n \circ C_m)$ is given in Table~{\em \ref{tab:P2-7Cm}} and $\gamma_a(P_n \circ P_m)$ is given in Table~{\em \ref{tab:P2-7Pm}}.
\end{proposition}

\begin{proof}
It is easy to check that $X_i^m$ is a GDA of $F = P_i \circ G^2$ for $i \in \{2, \ldots, 7\}$ and $G^2 \in \{C_m,P_m\}$. For the converse, let $S$ be a minimum GDA of $F$.

\bigskip
\noindent
Case $i = 2$.
For $m=3$, it is easy to check that there is no GDA of size 2 and that $V(G^2_1)$ is a GDA of the graph. Then assume $m \geq 4$. Since $V(G^2_1)$ is not a GDA, $(V(G^2_1) \setminus\{v_1\}) \cup \{v_2\}$ is not a a GDA for $v_1 \in V(G^2_2)$ and $v_2 \in V(G^2_2)$, and $x_2 \leq m$, it holds $2 \leq s_1 < m$ and $2 \leq s_2 < m$ for a minimum GDA $S$. Then, we can assume that there is a vertex $v \in S \cap V(G^2_1)$ such that $d(v) = m+2$ and having a neighbor in $V(G^2_1) \setminus S$. Therefore, we have $s_2 \geq \lfloor \frac{m+2}{2} \rfloor - 1 = \lfloor \frac{m}{2} \rfloor$. Since the same does hold for $s_1$, the result is true.

\bigskip
\noindent 
Case $i = 3$.
Suppose that $|S| < \cthreeC = x_3^m < 2m$. If $s_2 = 0$, then we can assume that $v \in V(G^2_1)$ has at most one neighbor in $S$, then $v$ is not defended in $S$. Hence $s_2 \geq 1$.
First, consider $m \in \{3,4,5\}$ and $G_2 \simeq C_m$. Since $d(v) = 2m+2$ for $v \in V(G^2_2)$, we have $|S \cap N(v)| < \frac{d(v)}{2}$, a contradiction.
Case $m \in \{3,4,5\}$ and $G_2 \simeq P_m$ is direct from Proposition~\ref{pro:basic}~$(ii)$.

Now, consider $m \geq 6$. We can now write $|S| < m + \mtwoovertwo$.
By Proposition~\ref{pro:basic}, $s_1 + s_3 \geq m-1$. Since $m - 1 > \mtwoovertwo$ for $m \geq 4$, we have $s_2 < m$. Consequently, by Proposition~\ref{pro:basic} again, we have $m \leq s_1 + s_3 < 2m$. Then, without loss of generality, there is a vertex in $V(G^2_1)$ having at most one neighbor in $S \cap V(G^2_1)$. Therefore $s_2 \geq \lfloor\frac{m+2}{2}\rfloor - 1 = \lfloor\frac{m}{2} \rfloor$, which means that $|S| \geq m + \lfloor\frac{m}{2} \rfloor > x_3$, a contradiction.

\bigskip
\noindent 
Case $i = 4$.
First, consider
$s_2 = 0$ and $s_3 = 0$. Then $m = 3$, $s_1 = s_4 = m$, and $|S| = 6 > x_4^3 = 5$. Next, consider $s_2 = 0$ and $s_3 \geq 1$. Then $m = 3$, $s_1 = m$ and, using Proposition~\ref{pro:basic}, $s_2 + s_3 + s_4 \geq m+2 = 5$. Which means $|S| \geq 8 > x_4^3 = 5$. The case $s_2 \geq 1$ and $s_3 = 0$ is analogue. 
Then $s_2 \geq 1$ and $s_3 \geq 1$. By Proposition~\ref{pro:lower}, $|S| \geq 2m-1$, completing the proof for $i = 4$.

\bigskip
\noindent 
Case $i = 5$. Since $x_5^m = 2m-1$ for $m \geq 5$, and $\gamma_a(P_5 \circ G^2) \geq \gamma_a(P_4 \circ G^2)$ for $G^2 \in \{C_m,P_m\}$, case $i=4$ implies, for $m\geq5$, $|S| \geq x_5^m$.
The same argument holds for $m \leq 4$ and $G^2 \simeq P_m$.

For $m \leq 4$ and $G^2 \simeq C_m$, we have $x_5^m = m + 4$. If $s_2 = 0$, then $m = 3$ and $s_1 = 3$. Since $s_3$ and $s_4$ are both not equal to 0, then $s_2 + s_3 + s_4 + s_5 \geq m+2$, which means $|S| > m + 4$. Then $s_2 \geq 1$ and $s_4 \geq 1$.
Suppose that $s_3 = m - k$ for $k \geq 1$. This implies that $s_1 + s_2 \geq 2 + k$ and $s_4 + s_5 \geq 2 + k$.
Then $|S| \geq m - k + 2 + k +2 + k = m + 4 + k$, a contradiction. Therefore $s_3 = m$. Since every vertex of $S_2 \cup S_4$ has degree $2m+2$, $s_1 + s_2 \geq 2$ and $s_4 + s_5 \geq 2$, which implies that $|S| \geq m + 4$.

\bigskip
\noindent 
Case $i = 6$. First, consider $m = 3$ and $G^2 \simeq C_m$. If $s_2 \geq 1$ and $s_5 \geq 1$, then $s_1 + s_2 + s_3 \geq 5$ and $s_4 + s_5 + s_6 \geq 5$, which means $|S| > x_6^3 = 8$. If $s_2 = 0$ and $s_5 = 0$, then $s_1 = s_6 = 3$, furthermore $s_3 + s_4 \geq 5$, which means $|S| > x_6^3$.
Then, without loss of generality, we can assume $s_2 = 0$ and $s_5 \geq 1$. This implies $s_1 = 3$ and $s_4 + s_5 + s_6 \geq 5$, which means $|S| \geq x_6^3 = 8$.

Next, consider $m = 3$ and $G^2 \simeq P_m$. Since $V(G^2_1)$ is not a GDA of $P_2 \circ P_3$, then $s_2 \geq 1$ and $s_5 \geq 1$.  Observe that a vertex $v \in S_2$ needs at least four neighbors in $S$ because $d(v) = 8$. Then $s_1 + s_2 + s_3 \geq 5$ and $s_4 + s_5 + s_6 \geq 5$.

Now, consider $m \geq 4$. It is clear that $s_2 \geq 1$ and $s_5 \geq 1$. By Proposition~\ref{pro:basic}, $s_1 + s_2 + s_3 \geq m+2$ for $G^2 \simeq C_m$ and $s_1 + s_2 + s_3 \geq m+1$ for $G^2 \simeq P_m$. By symmetry, $s_4 + s_5 + s_6 \geq m+2$ for $G^2 \simeq C_m$ and $s_4 + s_5 + s_6 \geq m+2$ for $G^2 \simeq P_m$.  Therefore $|S| \geq 2m+4$ for $G^2 \simeq C_m$ and $|S| \geq 2m+2$ for $G^2 \simeq P_m$.

\bigskip
\noindent 
Case $i = 7$.
If $s_2 = 0$, then $m = 3, G^2 \simeq C_m$, and $s_1 = 3$. If $s_3 \geq 1$, then $s_3 + s_4 \geq 5$. Since $s_5 + s_6 + s_7 \geq 3$, we have $|S| > x_7^3$. Then $s_3 = 0$. This implies that $s_4 \geq 1$, which means $s_4 + s_5 \geq 5$. Therefore $|S| < x_7^3$ if $s_6 + s_7 = 1$. But the vertex of $S_6 \cup S_7$ is not defendend in $S$.

Then, consider $s_2 \geq 1$ and $s_6 \geq 1$.
If $s'_3 = 0$, then cases $i = 3$ and $i = 5$ imply that $s_1 + s_2 + s_3 \geq m + \max\{2,\lfloor \frac{m-2}{2} \rfloor\}$ and $s_3 + s_4 + s_5 + s_6 + s_7 \geq \cfiveC$ which is at least $x_7^m$ for any $m \geq 3$. Then, we can consider $s_3 \geq 1$ and $s_5 \geq 1$.
Since $s_1+s_3$ and $s_2+s_4$ cannot both be equal to $m-1$ and 
for $G^2 \simeq C_m$, we have $s_5 + s_6 + s_7 \geq m+2$,
for $G^2 \simeq P_m$, we have $s_5 + s_6 + s_7 \geq m+1$, 
we have $|S| \geq m-1 + m + m + 2 = 3m + 1$ for $G^2 \simeq C_m$, and 
we have $|S| \geq m-1 + m + m + 2 = 3m$ for $G^2 \simeq P_m$.
\end{proof}

\begin{corollary} \label{cor:c2-7G2}
For $n \in \{3, \ldots, 7\}$, $\gamma_a(C_n \circ C_m)$ is given in Table~{\em \ref{tab:C2-7Cm}} and $\gamma_a(C_n \circ P_m)$ is given in Table~{\em \ref{tab:C2-7Pm}}.
\end{corollary}

\begin{proof}
Proposition~\ref{pro:lower} implies that $\gamma_a(C_3 \circ C_3) \geq 5$ and $\gamma_a(C_3 \circ P_3) \geq 4$. It is easy to check that $C_3 \circ P_3$ has no GDA with less than 5 vertices, then since $X_{3,3}$ and $Y_{3,3}$ are GDAs of $C_3 \circ C_3$ and $C_3 \circ P_3$, respectively, $\gamma_a(C_3 \circ C_3) = \gamma_a(C_3 \circ P_3) = 5$.

Now, consider $n = 3$, $m \geq 4$, and let $S$ be a minimum GDA of $C_n \circ G^2$ for $G^2 \in \{C_m,P_m\}$. We can assume that $v_2 \in S_2$ and $v_3 \in S_3$. Since $d(v_2) = d(v_3) \in \{2m+1, 2m+2\}$, $s_1 + s_3 \geq m-1$ and $s_1 + s_2 \geq m-1$. If $s_1 = 0$, then $|S| \geq 2m-2$. Therefore, we can assume $s_1 \neq 1$, which implies, by the symmetry of the graph, that $s_1 = s_2 = s_3 = k$. If $d(v_2) = 2m+1$ and $k < \lfloor \frac{m}{2} \rfloor$, then $S$ is not a GDA. Therefore $\gamma_a(C_3 \circ P_n) = \cthreeCP$ for $m \geq 4$. If $d(v_2) = 2m+2$ and $k < \lceil \frac{m}{2} \rceil$, then $S$ is not a GDA. Therefore $\gamma_a(C_3 \circ C_n) = \cthreeCC$ for $m \geq 4$. 

The cases $n \in \{4, 5, 6, 7\}$ are consequence of Propositons~\ref{pro:spanning} and~\ref{pro:p2-7G2}.
\end{proof}

\begin{corollary}
For $k \in \{3, \ldots, 7\}$, $G^1 \in \{C_n,P_n\}$, and $G^2 \in \{C_m,P_m\}$, there is a minimum GDA $S$ of $G^1 \circ G^2$ such that $max(spe(S)) \leq k$.
\end{corollary}

\subsection{Case $n \geq 8$} \label{sec:nbig}

We begin this section presenting a hyerarchy of $\gamma_a(G^1 \circ G^2)$ which depends of the operands and is consequence of the previous results.

\begin{corollary} \label{cor:hierarchy}
	For $n \geq 2$ and $m \geq 3$, it holds $\gamma_a(P_n \circ P_m) \leq \gamma_a(C_n \circ P_m) \leq  \gamma_a(C_n \circ C_m)$ and	$\gamma_a(P_n \circ P_m) \leq \gamma_a(P_n \circ C_m) \leq \gamma_a(C_n \circ C_m)$.
\end{corollary}

\begin{proof}
	$\gamma_a(P_n \circ P_m) \leq \gamma_a(C_n \circ P_m)$ and $\gamma_a(P_n \circ C_m) \leq \gamma_a(C_n \circ C_m)$ are consequences of Corollary~\ref{cor:dc}, while $\gamma_a(C_n \circ P_m) \leq  \gamma_a(C_n \circ C_m)$ and 
	$\gamma_a(P_n \circ P_m) \leq \gamma_a(P_n \circ C_m)$ are consequence of Corollaries~\ref{cor:dc}, ~\ref{cor:c2-7G2}, and Proposition~\ref{pro:p2-7G2}.
\end{proof}

Now, we consider the case where $G^1$ has order at least 8. We divide the study into two cases, $m = 3$ and $m \geq 4$.

\subsubsection{Case $m = 3$}

\begin{proposition} \label{pro:atmost6}
For $n \geq 8, G^1 \in \{P_n, C_n\}$, and $G^2 \in \{P_3, C_3\}$, there is minimum GDA $S$ of $G^1 \circ G^2$ such that $max(spe(S)) \leq 6$.
\end{proposition}

\begin{proof}
Write $F = G^1 \circ G^2$ and let $w = (k_1, \ldots, k_t)$ be the spectrum of a minimum GDA $S$ of $F$, $k_i \geq 7$ for some $i \in [t]$, and $r \equiv k_i \mod 4$.

By Proposition~\ref{pro:lower}~$(\ref{ite:m3})$, $\valE[C_3]{k_i} \geq 6 \lfloor \frac{k_i}{4} \rfloor + t$ where $t = \begin{cases}
r + 2 & \mbox{, if } r \in \{2,3\} \\
2 & \mbox{, if } r = 1 \\
0 & \mbox{, if } r = 0
\end{cases}$

For each value of $r$, we present a $k_i$-sequence $w' = (\ell_1, \ldots, \ell_p)$ such that $max(w') \leq 6$ and $\valT[P_n]{C_3}{w'} \leq 6 \lfloor \frac{k}{4} \rfloor + t$.
Since $y_{t,3} \leq x_{t,3}$ for $3 \leq t \leq 6$ and the bound of Proposition~\ref{pro:lower}~$(\ref{ite:m3})$ holds for $G^2 \simeq P_3$ and $G^2 \simeq C_3$, we only need to consider $G^2 \simeq C_3$.

For $r = 0$, define $w' = ([\frac{k}{4}]4)$ containing $\frac{k}{4}$. Since $x_{4,3} = 5$, it holds $\valT[P_n]{C_3}{w'} = 5\frac{k}{4} \leq 6 \frac{k}{4} \leq \valE[C_3]{k_i}$.
For $r = 1$, consider $w' = ([\frac{k-5}{4}]4,5)$. Since $x_{5,3} = 5$, it holds $\valT[P_n]{C_3}{w'} = 5\frac{k-5}{4} + 5 \leq 6 \lfloor\frac{k}{4}\rfloor + 2 \leq \valE[C_3]{k_i}$.
For $r = 2$, define $w' = ([\frac{k-6}{4}]4,6)$. Since $x_{6,3} = 8$, it holds $\valT[P_n]{C_3}{w'} = 5\frac{k-6}{4} + 8 \leq 6 \lfloor\frac{k}{4}\rfloor + 4 \leq \valE[C_3]{k_i}$.
For $r = 3$, define $w' = (3,[\frac{k-3}{4}]4)$. Since $val'(3) = x_{4,3} = 5$, it holds $\valT[P_n]{C_3}{w'} = 5\frac{k-3}{4} + 5 \leq 6 \lfloor\frac{k}{4}\rfloor + 5 \leq \valE[C_3]{k_i}$.

Now, it remains to observe that $w'' = w_{1,i-1}  w'  w_{i+1,t}$ is a feasible $n$-sequence and $\valT{C_3}{w''} \leq \valT{C_3}{w}$.
\end{proof}

\begin{proposition} \label{pro:tailthree}
For $n \geq 8$, $G^1 \in \{C_n,P_n\}$, and $G^2 \in \{C_3,P_3\}$, there is a minimim GDA $S$ of $G^1 \circ G^2$ such that

\begin{enumerate}[$(i)$]
\item if $G^2 \simeq C_3$, then $spe(S)$ has at most one element in the set $\{3,5,6\}$ and no one is equal to $2$; \label{ite:C}

\item if $G^2 \simeq P_3$, then $spe(S)$ has at most one element in the set $\{2,3,4,6\}$. \label{ite:P}
\end{enumerate}

\end{proposition}

\begin{proof}
By Proposition~\ref{pro:atmost6}, there is a minimum GDA $S$ of $F = G^1 \circ C_3$ whose $max(spe(S)) \leq 6$. Suppose that $k_i$ and $k_j$ are values of $spe(S)$ and of $\{2,3,5,6\}$. For each possible case, we present in Table~\ref{tab:atmostonem3C} a sequence $w' = (\ell_1, \ldots, \ell_t)$ for $t \leq 3$ such that $\valT[P_n]{C_3}{w'} \leq \valE[C_3]{k_i} + \valE[C_3]{k_j}$, $w'$ does not contain the number 2, and contains at most one element of the set $\{3,5,6\}$. The third column of the table is a lower bound of $\valE[C_3]{k_i} + \valE[C_3]{k_j}$, which is consequence of Proposition~\ref{pro:p2-7G2}.

\begin{table}[h]
	\centering
	\begin{tabular}{|l|l|l|l|l|l|l|l|}
		\hline
		$k_i$ & $k_j$ & $\valE[C_3]{k_i} + \valE[C_3]{k_j}$ & $\ell_1$ & $\ell_2$ & $\ell_3$ & $\valT[P_n]{C_3}{w'}$ \\ 
		\hline
		2 & 3 & $5 + 5 = 10$ & 5 & & & $7$ \\
		2 & 4 & $5 + 5 = 10$ & 6 & & & $8$ \\
		2 & 5 & $5 + 7 = 12$ & 4 & 3 & & $10$ \\
		2 & 6 & $5 + 8 = 13$ & 4 & 4 & & $10$ \\
		3 & 3 & $5 + 5 = 10$ & 6 & & & $8$ \\
		3 & 5 & $5 + 7 = 12$ & 4 & 4 & & $10$ \\
		3 & 6 & $5 + 8 = 13$ & 5 & 4 & & $12$ \\
		5 & 5 & $7 + 7 = 14$ & 6 & 4 & & $13$ \\
		5 & 6 & $7 + 8 = 15$ & 4 & 4 & 3 & $15$ \\
		6 & 6 & $8 + 8 = 16$ & 4 & 4 & 4 & $15$ \\
		\hline
	\end{tabular}
	\caption{Case $G^2 \simeq C_3$.}
	\label{tab:atmostonem3C}
\end{table}

It is clear that the sequence $w'' = (k_1, \ldots, k_{i-1}, k_{i+1}, \ldots, k_{j-1},k_{j+1}, \ldots, k_t, ) w'$ is feasible and $\valT{C_3}{w''} \leq \valT{C_3}{w}$. Since one can repeat this process until a sequence with the required properties be obtained, the result does hold.

The proof of $(ii)$ is essentially the same of $(i)$ by considering $G^2 \simeq P_3$ and Table~\ref{tab:atmostonem3P}.

\begin{table}[h]
	\centering
	\begin{tabular}{|l|l|l|l|l|l|l|}
		\hline
		$k_i$ & $k_j$ & $\valE[P_3]{k_i} + \valE[P_3]{k_j}$ & $\ell_1$ & $\ell_2$ & $\ell_3$ & $\valT[P_n]{P_3}{w'}$ \\ 
		\hline
		2 & 3 & $4 + 4 = 8$ & 5 & & & $5$ \\
		2 & 4 & $4 + 5 = 9$ & 6 & & & $8$ \\		
		2 & 6 & $4 + 8 = 12$ & 5 & 3 & & $10$ \\
		3 & 3 & $4 + 4 = 8$ & 6 & & & $8$ \\
		3 & 4 & $4 + 5 = 9$ & 5 & 2 & & $9$ \\
		3 & 6 & $4 + 8 = 12$ & 5 & 4 & & $10$ \\
		4 & 4 & $5 + 5 = 10$ & 5 & 3 & & $10$ \\
		4 & 6 & $5 + 8 = 13$ & 5 & 5 & & $10$ \\
		6 & 6 & $8 + 8 = 16$ & 5 & 5 & 2 & $14$ \\
		\hline
		
	\end{tabular}
	\caption{Case $G^2 \simeq P_3$.}
	\label{tab:atmostonem3P}
\end{table}

\end{proof}

$f(n,3) =
\begin{cases}
	5 \frac{n}{4} & \mbox{, if } n \equiv 0 \mod 4 \\
	5 \frac{n-5}{4} + 7 & \mbox{, if } n \equiv 1 \mod 4 \\
	5 \frac{n-6}{4} + 8 & \mbox{, if } n \equiv 2 \mod 4 \\
	5 \frac{n-3}{4} + 5 & \mbox{, if } n \equiv 3 \mod 4 
\end{cases}$

$f'(n,3) =
\begin{cases}
5 \frac{n}{5} & \mbox{, if } n \equiv 0 \mod 5 \\
5 \frac{n-6}{5} + 8 & \mbox{, if } n \equiv 1 \mod 5 \\
5 \frac{n-r}{5} + 4 & \mbox{, if } n \equiv r \mod 5 \mbox{ for } r \in \{2,3\} \\
5 \frac{n-4}{5} + 5 & \mbox{, if } n \equiv 4 \mod 5 
\end{cases}$

\begin{theorem} \label{the:G2m3}
For $n \geq 8$ and $G^1 \in \{C_n,P_n\}$, $\gamma_a(G^1 \circ C_3) = f(n,3)$ and $\gamma_a(G^1 \circ P_3) = f'(n,3)$.
\end{theorem}

\begin{proof}
Corollary~\ref{cor:external23} and Propositions~\ref{pro:atmost6} and~\ref{pro:tailthree}~$(\ref{ite:C})$ imply that, for $p = \lfloor \frac{n}{4} \rfloor$ and $r = n \mod 4$, it holds that a sequence $w$ such that $\gamma_a(G^1 \circ C_3) = \valT{C_3}{w}$ is

\bigskip
$w =
\begin{cases}
([p]4) & \mbox{, if } r = 0, \\
([p-1]4,5) & \mbox{, if } r = 1, \\
([p-1]4,6) & \mbox{, if } r = 2, \\
([p]4,3) & \mbox{, if } r = 3. \\
\end{cases}
$
\bigskip

Using Proposition~\ref{pro:p2-7G2}, we have $\gamma_a(G^1 \circ C_3) = f(n,3)$. Now, Corollary~\ref{cor:external23} and Propositions~\ref{pro:atmost6} and~\ref{pro:tailthree}~$(\ref{ite:P})$ imply that, for $p = \lfloor \frac{n}{4} \rfloor$ and $r = n \mod 5$, it holds that a sequence $w$ such that $\gamma_a(G^1 \circ P_3) = \valT{P_3}{w}$ is

\bigskip
$w =
\begin{cases}
([p-1]5) & \mbox{, if } r = 0, \\
([p-1]5, 6) & \mbox{, if } r = 1, \\
(r,[p]5) & \mbox{, if } r \in \{2,3,4\}.
\end{cases}
$
\bigskip

Using Proposition~\ref{pro:p2-7G2}, we have $\gamma_a(G^1 \circ P_3) = f'(n,3)$.
\end{proof}

\subsubsection{Case $m \geq 4$}

\begin{proposition} \label{pro:atmost7}
For $n \geq 8, m \geq 4$, $G^1 \in \{P_n, C_n\}$, and $G^2 \in \{P_m, C_m\}$, there is a minimim GDA $S$ of $G^1 \circ G^2$ such that $max(spe(S)) \leq 7$.
\end{proposition}

\begin{proof}
Let $w = (k_1, \ldots, k_t)$ be the spectrum of a minimum GDA $S$ of $F = G^1 \circ G^2$ such that $k_i \geq 8$ for some $i \in [t]$. Let $F'$ be the section of $F$ associated with $k_i$ and set $S' = |V(F') \cap S|$.
For each case, we present a $k_i$-sequence $w' = (\ell_1, \ldots, \ell_p)$ such that $max(w') \leq 7$ and $\valT[P_n]{G^2}{w'} \leq |S'|$.

For $k_i = 8$, consider $r \equiv k_i \mod 4$. Proposition~\ref{pro:lower}~$(\ref{ite:r0})$ implies $|S'| \geq 4m-2$. If $m = 4$, let $w' = (4,4)$. Since $y_{4,4} = x_{4,4} = 7$, $\valT[P_n]{G^2}{w'} = 14 \leq |S'|$. If $m \geq 5$, let $w' = (5,3)$. Since $y_{5,m} = x_{5,m} = 2m-1$ and $y_{3,m} \leq x_{3,m} = \cthreeC$, it holds $\valT[P_n]{G^2}{w'} \leq 3m - 1 + \cthreeCPtwo \leq 4m - 2 \leq |S'|$, which means that the result also holds for $k_i = 8$.

For $k_i \geq 9$, consider $r \equiv k_i \mod 5$.
Let $w'$ as follows

\bigskip
$w' =
\begin{cases}
([ \frac{k_i}{5}] 5) & \mbox{, if } r = 0 \\
([ \frac{k_i-6}{5}] 5,6) & \mbox{, if } r = 1 \\ 
(r,[ \frac{k_i-r}{5}] 5) & \mbox{, if } r \in \{2,3,4\}
\end{cases}$
\bigskip

Consider first $m = 4$. If $G^2 \simeq C_4$, Proposition~\ref{pro:lower}~$(\ref{ite:n9m4})$ implies $|S'| \geq 2k_i$. By Proposition~\ref{pro:p2-7G2}, it holds that $\valT[P_n]{C_4}{w'}$ is
$8 \frac{k_i}{5}$ for $r = 0$, is
$8 \frac{k_i-6}{5} + 12$ for $r = 1$, is
$8 \frac{k_i-r}{5} + 6$ for $r \in \{2,3\}$, is
$8 \frac{k_i-4}{5} + 7$ for $r = 4$. Since $\valT[P_n]{C_4}{w'} \leq 2k_i \leq |S'|$ in all cases, the result follows for $G^2 \simeq C_4$.
If $G^2 \simeq P_4$, Proposition~\ref{pro:lower}~$(\ref{ite:n9m4})$ implies $|S'| \geq 2k_i-2$. By Proposition~\ref{pro:p2-7G2}, it holds that $\valT[P_n]{P_4}{w'}$ is
$7 \frac{k_i}{5}$ for $r = 0$, is
$7 \frac{k_i-6}{5} + 10$ for $r = 1$, is
$7 \frac{k_i-r}{5} + 5$ for $r \in \{2,3\}$, is
$7 \frac{k_i-4}{5} + 7$ for $r = 4$.
Since $\valT[P_n]{P_4}{w'} \leq 2k_i-2 \leq |S'|$ in all cases, the result follows for $G^2 \simeq P_4$.

Consider now $m \geq 5$. Proposition~\ref{pro:lower}~$(\ref{ite:r0})$ and~$(\ref{ite:r123})$ imply $|S'| \geq \lfloor \frac{k}{4} \rfloor (2m-1) + t$
where

\bigskip
$t = 
\begin{cases}
m + 1 & \mbox{, if } r = 3 \\
r & \mbox{, if } r \in \{0,1,2\}
\end{cases}$
\bigskip

\noindent for $G^2 \in \{P_m, C_m\}$. By Proposition~\ref{pro:p2-7G2}, $\valT[P_n]{G^2}{w'}$ is 
$(2m-1) \frac{k_i}{5}$ for $r = 0$, is at most
$(2m-1) \frac{k_i-6}{5} + 2m+4$ for $r = 1$, is at most
$(2m-1) \frac{k_i-r}{5} + \cthreeC$ for $r \in \{2,3\}$, is
$(2m-1) \frac{k_i-4}{5} + 2m-1$ for $r = 4$.
Since $\valT[P_n]{G^2}{w'} \leq |S'|$ in all cases, the proof is complete.
\end{proof}

\begin{proposition} \label{pro:tail}
For $n \geq 8, m \geq 4, G^1 \in \{P_n, C_n\}$, and $G^2 \in \{P_m, C_m\}$, there is a minimum GDA $S$ of $G^1 \circ G^2$ whose spectrum contains at most one element of the set $\{2,3,4,7\}$.
\end{proposition}

\begin{proof}
By Proposition~\ref{pro:atmost7}, there is a minimum GDA $S$ of $F = G^1 \circ G^2$ such that $max(spe(S)) \leq 7$. Suppose that $k_i$ and $k_j$ are values of $spe(S)$ and of $\{2,3,4,7\}$. For each possible case, we present in Tables~\ref{tab:atmostoneC} and~\ref{tab:atmostoneP} a sequence $w' = (\ell_1, \ldots, \ell_t)$ for $t \leq 3$ such that $\valT[P_n]{G^2}{w'} \leq \valE{k_i} + \valE{k_j}$ such that $w'$ contains at most one element of the set $\{2,3,4,7\}$. Table~\ref{tab:atmostoneC} contains the cases for $G^2 \simeq C_m$ and Table~\ref{tab:atmostoneP} for $G^2 \simeq P_m$. The third column of each table contains a lower bound of $\valE{k_i} + \valE{k_j}$, which is consequence of Proposition~\ref{pro:p2-7G2}.
	
\begin{table}[h]
	\centering
	\begin{tabular}{|l|l|l|l|l|l|l|}
		\hline
		$k_i$ & $k_j$ & $\valE[C_m]{k_i} + \valE[C_m]{k_j}$ & $\ell_1$ & $\ell_2$ & $\ell_3$ & $\valT[P_n]{C_m}{w'}$ \\ 
		\hline
		2 & 3 & $2(\cthreeC)$ & 5 & & & $\cfiveC$ \\
		3 & 3 & $2(\cthreeC)$ & 6 & & & $\csixC$ \\
		2 & 4 & $\cfour + \cthreeC$ & 6 & & & $\csixC$  \\
		3 & 4 & $\cfour + \cthreeC$ & 7 & & & $\csevenC$  \\
		4 & 4 & $2(\cfour)$ & 5 & 3 & & $\cfiveC + \cthreeC$  \\
		2 & 7 & $\csevenC + \cthreeC$ & 5 & 4 & & $\cfiveC + \cfour$  \\
		3 & 7 & $\cthreeC + \csevenC$ & 5 & 5 & & $2(\cfiveC)$  \\
		4 & 7 & $\csevenC + \cfour$ & 6 & 5 & & $\csixC + \cfiveC$  \\
		7 & 7 & $2(\csevenC)$ & 5 & 5 & 4 & $2(\cfiveC) + \cfour$ \\
		\hline
	\end{tabular}
	\caption{Case $G^2 \simeq C_m$.}
	\label{tab:atmostoneC}
\end{table}

\begin{table}[h]
	\centering
	\begin{tabular}{|l|l|l|l|l|l|l|}
		\hline
		$k_i$ & $k_j$ & $\valE[P_m]{k_i} + \valE[P_m]{k_j}$ & $\ell_1$ & $\ell_2$ & $\ell_3$ & $\valT[P_n]{G^2}[P_m]{w'}$ \\ 
		\hline
		2 & 3 & $2(\cthreeP)$ & 5 & & & $\cfiveP$ \\
		3 & 3 & $2(\cthreeP)$ & 6 & & & $\csixP$ \\
		2 & 4 & $\cfour + \cthreeP$ & 6 & & & $\csixP$  \\
		3 & 4 & $\cfour + \cthreeP$ & 7 & & & $\csevenP$  \\
		4 & 4 & $2(\cfour)$ & 5 & 3 & & $\cfiveP + \cthreeP$  \\
		2 & 7 & $\csevenP + \cthreeP$ & 5 & 4 & & $\cfiveP + \cfour$  \\
		3 & 7 & $\cthreeP + \csevenP$ & 5 & 5 & & $2(\cfiveP)$  \\
		4 & 7 & $\csevenP + \cfour$ & 6 & 5 & & $\csixP + \cfiveP$ \\
		7 & 7 & $2(\csevenP)$ & 5 & 5 & 4 & $2(\cfiveP) + \cfour$ \\
		\hline
	\end{tabular}
	\caption{Case $G^2 \simeq P_m$.}
	\label{tab:atmostoneP}
\end{table}

It is clear that the sequence $w'' = (k_1, \ldots, k_{i-1}, k_{i+1}, \ldots, k_{j-1},k_{j+1}, \ldots, k_t, ) w'$ is feasible and $\valT{G^2}{w''} \leq \valT{G^2}{w}$ for $G \in \{P,C\}$. Since one can repeat this process until a sequence with the required properties be obtained, the result does hold.
\end{proof}

\begin{proposition} \label{pro:3sizes}
If $n \geq 8, m \geq 4$, $G^1 \in \{P_n,C_n\}$, $G^2 \in \{P_m,C_m\}$, and $w = (k_1, \ldots, k_t)$ is the spectrum of a minimum GDA of $G^1 \circ G^2$ containing three numbers that are  pairwise different, then we can assume that $w \in \{(3,[p-1]5,6),(3,5,[q-1]6)\}$ where $p = \frac{n-9}{5}$ and $q = \frac{n-8}{6}$.
\end{proposition}

\begin{proof}
Suppose that, for $i,j,r \in [t]$, $k_i,k_j,$ and $k_r$ are pairwise different.
By Proposition~\ref{pro:tail}, we can assume that $k_j = 5$ and $k_r = 6$.
In Tables~\ref{tab:3sizesC} and~\ref{tab:3sizesP}, we show that if $k_i \neq 3$, then there is a  $k_i$-sequence $w' = (\ell_1, \ldots, \ell_{t'})$ for $t' \leq 4$ such that $w'$ contains only numbers 3,5, and 6, and $\valT[P_n]{G^2}{w'} \leq \valE{k_i} + \valE{k_j} + \valE{k_r}$.

\begin{table}[h]
	\centering
	\begin{tabular}{|l|l|l|l|l|l|l|}
		\hline
		$k_i$ & $val(i,j,r)$ & $\ell_1$ & $\ell_2$ & $\ell_3$ & $\ell_4$ & $\valT[P_n]{C_m}{w'}$ \\ 
		\hline
		2  & $\cthreeC + \csixC + $ & 5 & 5 & 3 & & $2(\cfiveC) + $ \\
		&  $\cfiveC = $             & & & & & $\cthreeC =$ \\
		&  $5m + 3 + \cthreeCPtwo$  & & & & & $5m-2+ \cthreeCPtwo$ \\		  
		\hline
		4  & $\cfour + \cfiveC +$   & 5 & 5 & 5 & & $3(\cfiveC) = $ \\ 

		&  $\csixC \geq 6m +3$      & & & & & $6m-3$ \\
		\hline
		7  & $\cfiveC + \csixC + $  & 5 & 5 & 5 & 3 & For $m \leq 10$, \\		  
		&  $\csevenC = 7m + 4$      & & & & & $3(\cfiveC) + $ \\
		&  & & & & & $\cthreeC =$ \\
		&                           & & & & & $ 7m + \cthreeCPtwo -3$ \\		  
		&                           & 6 & 6 & 6 & & For $m \geq 11, 3(\csixC) = $\\
		&  & & & & & $6m+12$ \\
		\hline
	\end{tabular}
	\caption{Case $G^2 \simeq C_m$, where $val(i,j,r) = \valE[C_m]{k_i} + \valE[C_m]{k_j} + \valE[C_m]{k_r}$ }
	\label{tab:3sizesC}
\end{table}

\begin{table}[h]
	\centering
	\begin{tabular}{|l|l|l|l|l|l|l|}
	\hline
	$k_i$ & $val(i,j,r)$ & $\ell_1$ & $\ell_2$ & $\ell_3$ & $\ell_4$ & $\valT[P_n]{P_m}{w'}$ \\ 
		\hline
		2 &  $\cthreeP + \csixP + $ & 5 & 5 & 3 & & $2(\cfiveP) + \cthreeP =$ \\
		&    $\cfiveP = 5m + 1 + \lfloor \frac{m-2}{2} \rfloor$ & & & &             & $5m -2+ \lfloor \frac{m-2}{2} \rfloor$ \\
		\hline		
		4 &  $\cfour + \cfiveP + \csixP = $ & 5 & 5 & 5 & & $3(\cfiveP) = 6m-3$ \\ 
		&    $6m$ &  & & & &  \\
		\hline
		7 &  $\cfiveP + \csixP + $ & 5 & 5 & 5 & 3 & For $m \leq 10$, \\		  
		&    $\csevenP \geq 7m + 1$          &  & & & & $3(\cfiveP) + \cthreeP =$ \\		  
		&              &  & & & & $ 7m + \lfloor \frac{m-2}{2} \rfloor -3$ \\		  
		&              & 6 & 6 & 6 & & For $m \geq 11, 3(\csixP) = 6m+6$\\
		\hline
	\end{tabular}
	\caption{Case $G^2 \simeq P_m$, where $val(i,j,r) = \valE[P_m]{k_i} + \valE[P_m]{k_j} + \valE[P_m]{k_r}$}
	\label{tab:3sizesP}
\end{table}

It remains to prove that $w \neq (3, [p]5), [q]6)$ for $p,q \geq 2$. First, we consider $G^2 \simeq C_m$. We can assume that $w = (3,[2]5,[2]6,[p-2]5,[q-2]6)$. We know that $\valT[P_n]{C_m}{(3,[2]5,[2]6)} = 2(\csixC) + 2(\cfiveC) + \cthreeC= 9m + \cthreeCPtwo + 6$. For $m \leq 17$, $\valT[P_n]{C_m}{([5]5)} = 5(\cfiveC) = 10m-5$ and for $m \geq 18$, $\valT[P_n]{C_m}{([3]6,7)} =  3(\csixC) + \csevenC = 9m+13$, which means that $w$ is not the spectrum of a minimum GDA of $G^1 \circ C_m$.

Finally consider $G^2 \simeq P_m$.
We know that $\valT[P_n]{P_m}{(3,[2]5,[2]6)} = 2(\csixP) + 2(\cfiveP) + \cthreeP = 9m + 3 + \lfloor \frac{m-2}{2} \rfloor$.
For $m \leq 11$, $\valT[P_n]{P_m}{([5]5)} = 5(\cfiveP) = 10m-5$ and for $m \geq 12$, 
$\valT[P_n]{P_m}{([3]6,7)} = 3(\csixP) + \csevenP = 9m+6$, which means that $w$ is not the spectrum of a minimum GDA of $G^1 \circ P_m$.
\end{proof}

The above results reduce the number of sequences that can reach $\gamma_a(F)$ for $G^1 \circ G^2$, $G^1 \in \{C_n,P_n\}$, $G^2 \in \{C_m,P_m\}$, $n \geq 8$, and $m \geq 4$. In fact, we will show that, for a given $F$, $\gamma_a(F)$ can be determined considering at most four sequences, the ones defined in the sequel.

\bigskip
$f_{1,n} =
\begin{cases}
	([p]5) & \mbox{, if } n \equiv 0 \mod 5 \\
	(r,[p]5) & \mbox{, if } n \equiv r \mod 5 \mbox{ for } r \in \{2,3,4\}\\
	([p]5,6) & \mbox{, if } n \equiv 6 \mod 5 \\
\end{cases}$

\bigskip
$f_{2,n} =
\begin{cases}
	([p]5,[q]6) \mbox{ for maximum } p & \mbox{, if } n \neq 19 \\
	(3,[2]5,6) & \mbox{, if } n = 19 \\
\end{cases}$
	
\bigskip
$f_{3,n} = ([p]5,[q]6) \mbox{ for maximum } q$
	
\bigskip
$f_{4,n} =
\begin{cases}
	([q]6) & \mbox{, if }  n \equiv 0 \mod 6 \\
	(s, [q]6) & \mbox{, if }  n \equiv s \mod 6 \mbox{ for } s \in \{3,5\}\\

	([q]6,7) & \mbox{, if } n \equiv 1 \mod 6 \\
	
	(3,5,[q]6) & \mbox{, if } n \equiv 2 \mod 6 \\

	([2]5, [q]6) & \mbox{, if } n \equiv 4 \mod 6
\end{cases}$
\bigskip

For $i \in [4]$, $f_{i,n}$ is an infinite set of sequences, which is associated with at most one sequence if we fix the value of $n$. Therefore, when we can handle $f_{i,n}$ as a set.

\begin{theorem} \label{the:f1f4}
For $n \geq 8, m \geq 4, G^1 \in \{P_n,C_n\}$, and $G^2 \in \{P_m,C_m\}$, it holds $\gamma_a(G^1 \circ G^2) = \min \{\valT{G^2}{f_{1,n}}, \valT{G^2}{f_{2,n}}, \valT{G^2}{f_{3,n}}, \valT{G^2}{f_{4,n}}\}$.
\end{theorem}

\begin{proof}
Write $F = G^1 \circ G^2$. Corollary~\ref{cor:external23} and Propositions~\ref{pro:atmost7},~\ref{pro:tail},~\ref{pro:3sizes} imply that there is a sequence $w$ such that $ \valT{G^2}{w} = \gamma_a(F)$ and $w$ is a sequence of one of the following 8 sets of sequences for $G = C$ if $G^1 \simeq C_n$, $G = P$ if $G^1 \simeq P_n$, $p = \lfloor \frac{n}{5} \rfloor$, $q = \lfloor \frac{n}{6} \rfloor$, $r = n \mod 5$, and $s = n \mod 6$:

\begin{enumerate}[$T_1 =$]

	\item $\{([p]5)\}$, \label{c1n}
	\item $\{([p]5,r)$ for $r \in \{2,3,4\}\}$, \label{c1}
	\item $\{([q]6)\}$, \label{c2n}
	\item $\{([q]6,s)$ for $s \in \{2,3,4,5\}\}$, \label{c2}
	\item $\{([q-1]6,7)\}$, \label{c3}
	\item $\{(3,5,[q-1]6)\}$, \label{c4}
	\item $\{(3,[p-1]5,6)\}$, \label{c5}
	\item $\{([p']5,[q']6)$, for all positive integers $p'$ and $q'$ such that $5p' + 6q' = n\}$. \label{c6}
\end{enumerate}

We note that there are values of $n$ such that some of these sets are empty. Therefore, we need to show that, if $w$ belongs to some $T_i$ for $i \in [8]$ and $\valT{G^2}{w} = \gamma_a(F)$, then $w$ appears in some $f_{i,n}$, for $i \in [4]$.

\begin{itemize}
\item The sequences of $T_{\ref{c1n}}, T_{\ref{c1}}, T_{\ref{c2n}}, T_{\ref{c3}}$, and $T_{\ref{c4}}$ appear in $f_{1,n}, f_{1,n}, f_{4,n}, f_{4,n},$ and $f_{4,n}$, respectively, so there is nothing to do for these cases.

\item The sequences of $T_{\ref{c2}}$ appear in $f_{4,n}$ for $s \in \{3,5\}$. We will show: $(i)$ for $s \in \{2,4\}$, $\valT{G^2}{([q]6,s)} \leq \valT{G^2}{w}$ for some $w$ that appears in $f_{j,n}$ for some $j \in [4]$.

\item The 19-sequence of $T_{\ref{c5}}$ appears in $f_{2,n}$. We will show: $(ii)$ for $n \geq 22$, $\valT{G^2}{(3,[p-1]5,6)} \leq \valT{G^2}{w}$ for some $w$ that appears in $f_{j,n}$ for some $j \in [4]$.

\item Only two sequences of $T_{\ref{c6}}$ are considered, one in $f_{2,n}$ and the other in $f_{3,n}$. We will show: $(iii)$ only these two sequences of $T_{\ref{c6}}$ can reach the minimum.
\end{itemize}

Hence, to complete the proof it suffices to prove $(i)$, $(ii)$, and $(iii)$.

\bigskip
\noindent $(i)$
We show that $\valT{G^2}{w} \leq \valT{G^2}{w'}$ for $w \in T_{\ref{c4}}$ and $w' \in T_{\ref{c2}}$ with $s = 2$. First, consider $G^2 \simeq C_m$.
For $G^1 \simeq P_n$, suppose that 
$qx_{6,m}+x_{3,m} < (q-1)x_{6,m} + x_{5,m}+x_{3,m}$ for some $n$. We have
$q(\csixC) + \cthreeC < (q-1)(\csixC) + \cfiveC + \cthreeC$. Then $\csixC < \cfiveC$, a contradiction.
For $G^1 \simeq C_n$, suppose that $qx_{6,m}+x_{4,m} < (q-1)x_{6,m} + x_{5,m}+x_{4,m}$ for some $n$. We have
$q(\csixC) < (q-1)(\csixC) + \cfiveC \Rightarrow \csixC < \cfiveC$, a contradiction.

Now, consider $G^2 \simeq P_m$.
For $G^1 \simeq P_n$, suppose that 
$qy_{6,m} + y_{3,m} < (q-1) y_{6,m} + y_{5,m} + y_{3,m}$ for some $n$. We have
$q(\csixP) + \cthreeP < (q-1)(\csixP) + \cfiveP + \cthreeP$. Then $\csixP < \cfiveP$, a contradiction.
For $G^1 \simeq C_n$, suppose that $qy_{6,m}+y_{4,m} < (q-1)y_{6,m} + y_{5,m} + y_{4,m}$ for some $n$. We have
$q(\csixP) < (q-1)(\csixP) + \cfiveP \Rightarrow \csixP < \cfiveP$, a contradiction.

Next, we show that $\valT{G^2}{([2]5,[q-1]6)} \leq \valT{G^2}{([q]6,4)}$. Suppose that $qx_{6,m}+x_{4,m} < (q-1)x_{6,m}+2x_{5,m}$ for some $n$. We have $q(\csixC) + \cfour < (q-1)(\csixC) + 2(\cfiveC)$. Then $4m+3 < 2(\cfiveC)$,
For $m = 3$, $15 < 14$; $m = 4$, $19 < 16$; $m = 5$, $23 < 18$, a contradiction.

\bigskip
\noindent $(ii)$ We show that $\valT{G^2}{w_{\ref{c5}}} \geq \min \{\valT{G^2}{w_{\ref{c1}}}, \valT{G^2}{w_{\ref{c6}}}\}$ where $w_{\ref{c5}} =(3, [p-1]5, 6) \in T_{\ref{c5}}$, $w_{\ref{c1}} = ([p]5,4) \in T_{\ref{c1}}$, and $w_{\ref{c6}} = ([p'5],[q']6) \in T_{\ref{c6}}$ for $n \geq 24$. Since the 3 sequences have a $24$-subsequence, we do the analysis comparing the correspoding $24$-subsequences $w'_{\ref{c5}}, w'_{\ref{c1}}$, and $w'_{\ref{c6}}$.
First consider $G^1 \simeq P_n$ and $G^2 \simeq C_m$.
If $m \leq 5$, $4x_{5,m}+x_{4,m} = 6m + 15$ while $3x_{5,m}+x_{6,m}+x_{3,m} = 3(m+4) + \csixC + \cthreeC = 6m + 16 + \cthreeCPtwo$.
Then for $m = 4$, $\valT[P_n]{C_m}{w'_{\ref{c1}}} = 39 < \valT[P_n]{C_m}{w'_{\ref{c5}}} = 42$; and for $m = 5$, $\valT[P_n]{C_m}{w'_{\ref{c1}}} = 45 \leq \valT[P_n]{C_m}{w'_{\ref{c5}}} = 48$. If $m \geq 6$, $3x_{5,m}+x_{6,m}+x_{3,m} = 6m-3 + \csixC + m + \mtwoovertwo = 9m + 1 + \mtwoovertwo$ and $4x_{6,m} = 4(\csixC) = 8m + 8$, which means  that $\valT[P_n]{C_m}{w'_{\ref{c6}}} = 8m + 8 < 9m + 1 + \mtwoovertwo = \valT[P_n]{C_m}{w'_{\ref{c5}}}$ for $m \geq 6$.
The result for $G^1 \simeq C_n$ and $G^2 \simeq C_m$ is consequence of the fact that $\valT[C_n]{C_m}{w_{\ref{c5}}} \geq \valT[P_n]{C_m}{w_{\ref{c5}}}$ is true due Corollary~\ref{cor:hierarchy}. It remains to consider $G^1 \in \{C_n,P_n\}$ and $G^2 \simeq P_m$. Now, it suffices to observe that $\valT{P_m}{w'_{\ref{c1}}} = 10m - 5 < \valT{P_m}{w'_{\ref{c5}}} = 10m -2$ for every $m \geq 4$.

\bigskip
\noindent $(iii)$
Suppose that the minimum one is achieved by $([p']5,[q']6)$ for $p' < p$ and $q' < q''$. (We consider $p$ maximum and $q''$ maximum). This means that we can change either $([6]5)$ for $([5]6)$ or vice-versa obtaining a smaller GDA, a contradiction.
\end{proof}

Theorems~\ref{the:G2m3} and~\ref{the:f1f4} lead to a constant-time algorithm for computing $\gamma_a(G^1 \circ G^2)$ for $G^1 \in \{C_n,P_n\}$, $G^2 \in \{C_m,P_m\}$, $n \geq 8$ and $m \geq 3$. It consists in computing at most four values and choosing the minimum one. In the next section, we show that functions $f_{k,n}$ have an homogeneous behavior, which allows one to characterize, for each pair $\{n,m\}$, which function gives the global defensive alliance number of $G^1 \circ G^2$.

\section{Deepening the results} \label{sec:imp}

It is easy to verify that if $n \geq 8$ is such that $f_{i,n}$ and $f_{i+1,n}$ are defined, then there is an integer $m_0$ such that $\valT{G^2}{f_{i,n}} \geq \valT{G^2}{f_{i,n}}$ for $G^2 \in \{C_m,P_m\}$ and $m \geq m_0$. The minimum $m_0$ with this property is the {\em threshold between $f_{i,n}$ and $f_{i+1,n}$} and will be denoted by $t_{n,i}$. If one of the functions is not defined or if $\valT{G^2}{f_{i,n}} = \valT{G^2}{f_{i,n}}$ for every $m$ that both functions are defined, we will say that $t_{n,i}$ is undefined.

\begin{proposition} \label{pro:tn23}
If $t_{n,2}$ is defined for $n$, then $t^{CC}_{n,2} = t^{PC}_{n,2} = 13$ and $t^{CP}_{n,2} = t^{PP}_{n,2} = 8$.
\end{proposition}

\begin{proof}
Let $w_2 = ([p]5,[q]6) \in f_{2,n}$ and $w_3 = ([p']5,[q']6) \in f_{3,n}$. If $\valT{G^2}{w_2} \neq \valT{G^2}{w_3}$, then $p > p'$. Furthermore,
$p = 6k + p'$ and $q' = 5k + q$ for $k \geq 1$.

Since $\valI{k} = \valE{k}$ for $G^2 \in C_m$ and $k \in \{5,6\}$, we have
$\valT{G^2}{w_2} = p\times\valI{5} + q \times \valI{6}$ and
$\valT{G^2}{w_3} = p'\times\valI{5} + q' \times \valI{6}$. Replacing, we have $\valT{G^2}{w_2} = (6k + p')\valI{5} + (q'-5k) \valI{6} = 6k \times \valI{5} - 5k \times \valI{6} + \valT{G^2}{w_3}$.

For $G^2 \simeq C_m$, we have $\valT{C_m}{w_2} = 6k(2m-1) - 5k(\csixC) + \valT{C_m}{w_3} = k(2m-26) + \valT{C_m}{w_3}$, which meanst that $t^{PC}_{n,2} = t^{CC}_{n,2} = 13$.

For $G^2 \simeq P_m$, we have $\valT{P_m}{w_2} = 6k(\cfiveP) - 5k(\csixP) + \valT{P_m}{w_3} = k(2m-16) + \valT{P_m}{w_3}$, which meanst that $t^{PP}_{n,2} = t^{CP}_{n,2} = 8$.
\end{proof}

\begin{proposition}
For every $n$ and $m \geq 4$, $t_{n,1}$ is given in Table~{\em \ref{tab:tn1}}.
\end{proposition}

\begin{table}[h]
	\centering
	\begin{tabular}{|l|l|l|l|l|}
		\hline
		$n \mod 5$ & $t^{PC}_{n,1}$ & $t^{CC}_{n,1}$ & $t^{PP}_{n,1}$ & $t^{CP}_{n,1}$ \\ 
		\hline
		0 & 13 & 13 & 8 & 8 \\
		1 & $\ast$ & $\ast$ & $\ast$ & $\ast$ \\
		2 & 8 & 6 & 5 & 4 \\
		3 & 9 & 8 & 7 & 5 \\
		$4, n \neq 19$ & 11 & 11 & 7 & 7 \\
		$n = 19$ & 9 & $\ast$ & 5 & $\ast$ \\
		\hline
	\end{tabular}
	\caption{$t_{n,1}$.}
	\label{tab:tn1}
\end{table}

\begin{proof}
\noindent {\bf Case 1} ($n \equiv 1 \mod 5, f_{1,n} =f_{2,n}$)

\bigskip
\noindent{\bf Case 2} ($n \equiv 2 \mod 5, f_{1,n} = (2,[p]5)$, and $f_{2,n} = ([p-2]5,[2]6)$)

For $P_n \circ C_m$, using Corollary~\ref{cor:dc}, $\valT[P_n]{C_m}{f_{1,n}} = px_{5,m}+x_{3,m} \geq \valT[P_n]{C_m}{f_{2,n}} = (p-2)x_{5,m}+2x_{6,m}$.

\[2(\cfour) + m + \mtwoovertwo \geq 2(\csixC) \]

\[5m - 2 + \mtwoovertwo \geq 4m + 8 \Rightarrow m + \mtwoovertwo \geq 10 \]

\noindent which is true for $m \geq 8$.

For $C_n \circ C_m$, using Corollary~\ref{cor:dc},
$\valT[C_n]{C_m}{f_{1,n}} = px_{5,m}+x_{4,m} \geq \valT[C_n]{C_m}{f_{2,n}} = (p-2)x_{5,m}+2x_{6,m}$

\[p(\cfour) + \cfour \geq (p-2)(\cfour) + 2(\csixC)\]

\[3(\cfour) \geq 2(\csixC) \Rightarrow 6m - 3 \geq 4m + 8 \Rightarrow 2m \geq 11\]

\noindent that is true for $m \geq 6$.

For $P_n \circ P_m$, using Corollary~\ref{cor:dc},
$\valT[P_n]{P_m}{f_{1,n}} = py_{5,m}+y_{3,m} \geq \valT[P_n]{P_m}{f_{2,n}} = (p-2)y_{5,m}+2y_{6,m}$.

\[2(\cfour) + \cthreeP \geq 2(\csixP) \]

\[m + \mtwoovertwo \geq 6\]

\noindent which is true for $m \geq 5$.

For $C_n \circ P_m$, using Corollary~\ref{cor:dc},
$\valT[C_n]{P_m}{f_{1,n}} = py_{5,m}+y_{4,m} \geq \valT[C_n]{P_m}{f_{2,n}} = (p-2)y_{5,m}+2y_{6,m}$.

\[2(\cfour) + \cfour \geq 2(\csixP) \Rightarrow 2m \geq 7\]

\noindent which is true for $m \geq 4$.

\bigskip
\noindent{\bf Case 3} ($n \equiv 3 \mod 5, f_{1,n} = (3,[p]5)$, and $f_{2,n} = ([p-3]5,[3]6)$)

For $P_n \circ C_m$,
$\valT[P_n]{C_m}{f_{1,n}} = px_{5,m} + x_{3,m} \geq \valT[P_n]{C_m}{f_{2,n}} = (p-3)x_{5,m} + 3x_{6,m}$.

\[ 3(\cfiveC) + \cthreeC \geq 3(\csixC)\]

\noindent which is true for $m \geq 9$.

For $C_n \circ C_m$,
$\valT[C_n]{C_m}{f_{1,n}} = px_{5,m} + x_{4,m} \geq \valT[C_n]{C_m}{f_{2,n}} = (p-3)x_{5,m} + 3x_{6,m}$.

\[ 3(\cfiveC) + \cfour \geq 3(\csixC)\]

\noindent which is valid for $m \geq 8$.

For $P_n \circ P_m$,
$\valT[P_n]{P_m}{f_{1,n}} = py_{5,m} + y_{3,m} \geq \valT[P_n]{P_m}{f_{2,n}} = (p-3)y_{5,m} + 3y_{6,m}$.

\[ 3(\cfiveP) + \cthreeP \geq 3(\csixP)\]

\[\cthreeP \geq 9\]

\noindent which is true for $m \geq 7$.

For $C_n \circ P_m$,
$\valT[C_n]{P_m}{f_{1,n}} = py_{5,m} + y_{4,m} \geq \valT[C_n]{P_m}{f_{2,n}} = (p-3)y_{5,m} + 3y_{6,m}$.

\[ 3(\cfiveP) + \cfour \geq 3(\csixP)\]

\[2m \geq 10\]

\noindent which is true for $m \geq 5$.

\bigskip
\noindent{\bf Case 4} ($n \equiv 4 \mod 5, n \neq 19, f_{1,n} = (4,[p]5)$, and $f_{2,n} = ([p-4]5,[4]6)$)

For $P_n \circ C_m$ and $C_n \circ C_m$, we have $4x_{5,m} + x_{4,m} \geq 4x_{6,m}$

\[4(\cfiveC) + \cfour \geq 4(\csixC)\]

\noindent which is true for $m \geq 11$.

For $P_n \circ P_m$ and $C_n \circ P_m$, we have $4y_{5,m} + y_{4,m} \geq 4y_{6,m}$

\[4(\cfiveP) + \cfour \geq 4(\csixP) \Rightarrow 2m \geq 13\]

\noindent which is true for $m \geq 7$.

\bigskip
\noindent{\bf Case 5} ($n \equiv 0 \mod 5, f_{1,n} = ([p]5)$, and $f_{2,n} = ([p-6]5,[5]6)$)

For $P_n \circ C_m$ and $C_n \circ C_m$, we have

\[6x_{5,m} \geq 5x_{6,m} \Rightarrow 6(\cfour) \geq 5(\csixC) \Rightarrow 12m - 6 \geq 10m +20 \Rightarrow 2m \geq 26 \]

\noindent which is true for $m \geq 13$.

For $P_n \circ P_m$ and $C_n \circ P_m$, we have

\[6y_{5,m} \geq 5y_{6,m} \Rightarrow 6(\cfiveP) \geq 5(\csixP) \Rightarrow 12m - 6 \geq 10m + 10\]

\noindent which is true for $m \geq 8$.

\bigskip
\noindent{\bf Case 6} ($n = 19, f_{1,n} = (4,[3]5)$, and $f_{2,n} = (3,[2]5,6)$)

For $P_n \circ C_m$,

\[3(\cfiveC) + \cfour \geq 2(\cfour) + \csixC + m + \mtwoovertwo \]

\[ 8m - 4 \geq 7m +2 + \mtwoovertwo \Rightarrow m \geq 6 + \mtwoovertwo \]

\noindent that is true for $m \geq 9$.

For $C_n \circ C_m$, $\valT[C_n]{C_m}{f_{1,n}} = 3x_{5,m}+x_{4,m} \geq \valT[C_n]{C_m}{f_{2,n}} = 2x_{5,m}+x_{6,m}+x_{4,m}$.

\[\cfiveC \geq \csixC \]

\noindent since there is no $m$ satisfying the above inequality, $t^{CP}_{19,1}$ is not defined.

For $P_n \circ P_m$, 

\[3(\cfiveP) + \cfour \geq 2(\cfiveP) + \csixP + \cthreeP \]

\[8m - 4 \geq 7m + \mtwoovertwo \Rightarrow m - \mtwoovertwo \geq 4 \]

\noindent that is true for $m \geq 5$.

For $C_n \circ P_m$, 

\[3(\cfiveP) + \cfour \geq 2(\cfiveP) + \csixP + \cfour \]

\[8m - 4 \geq 8m -1 \]

\noindent since there is no $m$ satisfying the above inequality, $t^{CC}_{19,1}$ is not defined.
\end{proof}

\begin{proposition} \label{pro:tn34}
For $n \geq 8$ and $m \geq 4$, $t_{n,3}$ is given in Table~{\em \ref{tab:tn3}}.	
\end{proposition}

\begin{table}[h]
	\centering
	\begin{tabular}{|l|l|l|l|l|}
		\hline
		$n \mod 6$ & $t^{PC}_{n,3}$ & $t^{CC}_{n,3}$ & $t^{PP}_{n,3}$ & $t^{CP}_{n,3}$ \\ 
		\hline
		0 & $\ast$ & $\ast$ & $\ast$ & $\ast$ \\		  
		1 & 18 & 18 & 11 & 11 \\ 
		2 & 19 & $\ast$ & 6 & $\ast$ \\
		3 & 19 & $\ast$ & 6 & $\ast$ \\ 
		4 & $\ast$ & $\ast$ & $\ast$ & $\ast$ \\ 
		5 & $\ast$ & $\ast$ & $\ast$ & $\ast$ \\ 
		\hline
	\end{tabular}
	\caption{$t_{n,3}$.}
	\label{tab:tn3}
\end{table}

\begin{proof}
\bigskip
\noindent{\bf Case 1} ($n \equiv 1 \mod 6, f_{3,n} = ([5]5,[q-4]6)$, and $f_{4,n} = ([q-1]6,7)$)

For $P_n \circ C_m$ and $C_n \circ C_m$,
$5x_{5,m} \geq 3x_{6,m} + x_{7,m}$.

\[ 5(\cfiveC) \geq 3(\csixC) + \csevenC\]

From, $10m - 5 \geq 9m +13$, we have that $t_{n,3} = 18$.

For $P_n \circ P_m$ and $C_n \circ P_m$, we can write $5y_{5,m} \geq 3y_{6,m} + y_{7,m}$. Thus

\[ 5(\cfiveP) \geq 3(\csixP) + \csevenP\]

From, $10m - 5 \geq 9m +6$, we have that $t_{n,3} = 11$.

\bigskip
\noindent{\bf Case 2} ($n \equiv 2 \mod 6, f_{3,n} = ([4]5,[q-3]6)$, and $f_{4,n} = (3,5,[q-1]6)$)

For $P_n \circ C_m$, we have $4x_{5,m} \geq 2x_{6,m} + x_{5,m} + x_{3,m}$.

$3(\cfiveC) \geq 2(\csixC) + \cthreeC$. Thus

\[6m -3 \geq 5m + 8 + \mtwoovertwo \Rightarrow m \geq 11 + \mtwoovertwo\]

\noindent is true for $m \geq 19$.

For $C_n \circ C_m$, we have $4x_{5,m} \geq 2x_{6,m} + x_{5,m} + x_{4,m}$.

Since there is no positive $m$ satisfying $3(\cfiveP) \geq \cfiveP + \cfour$, $t_{n,3}$ is undefined for this case.

For $P_n \circ P_m$, we have $4y_{5,m} \geq 2y_{6,m} + y_{5,m} + y_{3,m}$.

$3(\cfiveP) \geq 2(\csixP) + \cthreeP$. Then

\[6m -3 \geq 5m + 4 + \mtwoovertwo \Rightarrow m \geq 7 + \mtwoovertwo\]

\noindent is true for $m \geq 6$.

For $C_n \circ P_m$, we have $4y_{5,m} \geq 2y_{6,m} + y_{5,m} + y_{4,m}$.

$3(\cfiveP) \geq 2(\csixP) + \cfour$. Since $6m -3 \geq 6m + 3$ is not true for any positive $m$, $t_{n,3}$ is undefined for this case.

\bigskip
\noindent{\bf Case 3} ($n \equiv 3 \mod 6, f_{3,n} = ([3]5,[q-2]6)$, and $f_{4,n} = (3,[q]6)$)

For $P_n \circ C_m$, one has $3x_{5,m} \geq 2x_{6,m} + x_{3,m} \Rightarrow 3(\cfiveC) \geq 2(\csixC) + \cthreeC$

\[6m - 3 \geq 4m + 8 + m + \mtwoovertwo \Rightarrow m \geq 11 + \mtwoovertwo \]

\noindent that is true for $m \geq 19$.

For $C_n \circ C_m$, one has $3x_{5,m} \geq 2x_{6,m} + x_4\Rightarrow 3(\cfiveC) \geq 2(\csixC) + \cfour$. From $6m - 3 \geq 6m + 7$, we conclude that $t_{n,3}$ is undefined for this case.

For $P_n \circ P_m$, one has $3y_{5,m} \geq 2y_{6,m} + y_{3,m} \Rightarrow 3(\cfiveP) \geq 2(\csixP) + \cthreeP$.

\[6m - 3 \geq 5m + 4 + \mtwoovertwo \Rightarrow m + \mtwoovertwo \geq 7\]

\noindent which is true for $m \geq 6$.

For $C_n \circ P_m$, one has $3y_{5,m} \geq 2y_{6,m} + y_{4,m} \Rightarrow 3(\cfiveP) \geq 2(\csixP) + \cfour$.

Since $6m - 3 \geq 6m + 3$ is not true for any positive $m$, $t_{n,3}$ is undefined for this case.

\bigskip
\noindent{\bf Case 4} ($n \equiv 4 \mod 6, f_{3,n} = ([2]5,[q-1]6)$, and $f_{4,n} = ([2]5,[q-1]6)$)

Since $f_{n,3} = f_{n,4}$, $t_{n,3}$ is undefined for this case.

\bigskip
\noindent{\bf Case 5} ($n \equiv 5 \mod 6, f_{3,n} = (5,[q]6)$, and $f_{4,n} = (5,[q]6)$)

Since $f_{n,3} = f_{n,4}$, $t_{n,3}$ is undefined for this case.

\bigskip
\noindent{\bf Case 6} ($n \equiv 0 \mod 6, f_{3,n} = ([q]6)$, and $f_{4,n} = ([q]6)$).

Since $f_{n,3} = f_{n,4}$, $t_{n,3}$ is undefined for this case.
\end{proof}

\begin{corollary} \label{cor:n8m4}
For $n \geq 8, m \geq 4, G^1 \in \{C_n,P_n\},$ and $G^2 \in \{C_m,P_m\}$, it holds

$\gamma_a(G^1 \circ G^2)  = 
\begin{cases}
\valT{G^2}{f_{1,n}} & \mbox{, if } m < \min\{t_{n,1}, t_{n,2}\} \\
\valT{G^2}{f_{2,n}} & \mbox{, if } t_{n,1} \mbox{ is defined and } t_{n,1} \leq m \leq t_{n,2} \\
\valT{G^2}{f_{3,n}} & \mbox{, if } t_{n,3} \mbox{ is defined and } t_{n,2} \leq m < t_{n,3} \\
\valT{G^2}{f_{4,n}} & \mbox{, if } m \geq \max\{t_{n,3}, t_{n,2}\}
\end{cases}$	
\end{corollary}

\begin{proof}
For $m \geq 4$, the result is consequence of Theorem~\ref{the:f1f4} and Propositions~\ref{pro:tn23} to~\ref{pro:tn34}.
\end{proof}

\section{Conclusion} \label{sec:conclusion}

One can determining the global defensive alliance number of a graph $F = G^1 \circ G^2$ for $G^1 \in \{C_n,P_n\}$ and $G^2 \in \{C_m,P_m\}$ within a constant number of arithmetic operations.

For $n \leq 7$, the answer is obtained directly from Tables~\ref{tab:P2-7Cm} to~\ref{tab:C2-7Pm}. For instance, $\gamma_a(P_5 \circ C_3) = 7$ due Proposition~\ref{pro:p2-7G2} and $\gamma_a(C_5 \circ P_3) = 5$ due Proposition~\ref{cor:c2-7G2}.

For $n \geq 8$, consider as an example $P_{20} \circ C_{15}$. Since $t^{PC}_{2,3} = 13$ (Proposition~\ref{pro:tn23}) and $t^{PC}_{2,3} = 19$ (Proposition~\ref{pro:tn34}), Corollary~\ref{cor:n8m4}, implies that $\gamma_a(P_{20} \circ C_{15}) = f_{3,20} = \valT[P_n]{C_{15}}{([4]5)} = 4 x_{5,15} = 116$.
As another example, consider the graph $C_{20} \circ P_{15}$. Since $t^{CP}_{2,3} = 8$ (Proposition~\ref{pro:tn23}) and $t^{CP}_{2,3}$ is undefined (Proposition~\ref{pro:tn34}), Corollary~\ref{cor:n8m4}, implies that $\gamma_a(C_{20} \circ P_{15}) = f_{4,20} = \valT[C_n]{P_{15}}{(3,5,[2]6)} = y_{4,15} + y_{5,15} + 2 y_{6,15} = 29 + 29 + 2*32 = 122$.

For concluding, we remark that the four examples presented in this section show that the only relation not contained in Corollary~\ref{cor:hierarchy} indeed cannot be stablished because $\gamma_a(P_5 \circ C_3) = 7 > 5 = \gamma_a(C_5 \circ P_3)$ and 
$\gamma_a(P_{20} \circ C_{15}) = 116 < 122 = \gamma_a(C_{20} \circ P_{15})$.



\begin{thebibliography}{10}

\bibitem{Ber2010}
S. Bermudo, J. A. Rodríguez-Velázquez, J. M. Sigarreta, and I. G. Yero.
On global offensive k-alliances in graphs. Applied Mathematics Letters,
2010.

\bibitem{Bri2004}
R. C. Brigham, R. D. Dutton, and S. T. Hedetniemi. A sharp lower
bound on the powerful alliance number of $c_m \times c_n$ . Congressus Numer-
antium, 167:57–63, 2004.

\bibitem{Bri2009}
R. C. Brigham, R. D. Dutton, and S. T. Hedetniemi. Powerful alliance
in graphs. Discrete Mathematics, 309(8):2140–2147, 2009.

\bibitem{Cam2006}
Aurel Cami, Hemant Balakrishnan, Narsingh Deo, and Ronald D. Dut-
ton. On the complexity of finding optimal global alliances. JCMCC.
The Journal of Combinatorial Mathematics and Combinatorial Computing, 58(4):23–31, 2006.

\bibitem{Cha2012}
Chan-Wei Chang, Ma-Lian Chia, Cheng-Ju Hsu, David Kuo, Li-Ling
Lai, and Fu-Hsing Wang. Global defensive alliances of trees and carte-
sian product of paths and cycles. Discrete Applied Mathematics, 160(4-
5):479–487, March 2012.

\bibitem{Dou2011}
MC Dourado, LD Penso, D Rautenbach, and JL Szwarcfiter. The south
zone: distributed algorithms for alliances. In Stabilization, Safety, and
Security of Distributed Systems, pages 178–192. Springer, 2011.

\bibitem{Fav2004}
O. Favaron, G. Fricke, W. Goddard, S. M. Hedetniemi, S. T. Hedet-
niemi, P. Kristiansen, R. C. Laskar, and R. D. Skaggs. Offensive al-
liances in graphs. Discussiones Mathematicae Graph Theory, 24(2):263–
275, 2004.

\bibitem{Hed2004}
P. Kristiansen and S. M. Hedetniemi and S. T. Hedetniemi, Alliances in graphs, JCMCC The Journal of Combinatorial Mathematics and Combinatorial Computing, 48 (2004), 157--177.

\bibitem{HHH2003}
T.W. Haynes, S.T. Hedetniemi, M.A. Henning, Global defensive alliances in graphs, Electron. J. Combin. 10 (2003) R47.

\bibitem{Fer2014}
Henning Fernau and Juan A. Rodr'iguez-Velázquez. A survey on alliances and related parameters in graphs. Eletronic Journal of Graph
Theory and Applications, 2(1):70 – 86, 2014.

\bibitem{Fer2008}
Henning Fernau, Juan A. Rodríguez-Velázquez, and José M. Sigarreta. Offensive r-alliances in graphs. Discrete Applied Mathematics,
157(1):177 – 182, 2008.

\bibitem{Rod2006}
J. A. Rodríguez-Velázquez and J. M. Sigarreta. Global offensive al-
liances in graphs. Discrete Mathematics, 25:157–164, 2006.

\bibitem{Rod2009}
J.A. Rodríguez-Velázquez and J. M. Sigarreta. Global defensive k-alliances in graphs. {\em Discrete Applied Mathematics}, 157:211--218, 2009.

\bibitem{Rod2009b}
J. A. Rodríguez-Velázquez and J. M. Sigarreta. Defensive k-alliances
in graphs. Applied Mathematics, 22:96–100, 2009.

\bibitem{Sha2002}
K. H. Shafique and R. D. Dutton. On satisfactory partitioning of
graphs. Congressus Numeratium, 154:183–194, 2002.

\bibitem{Sha2003}
K. H. Shafique and R. D. Dutton. Maximum alliance-free and minimum
alliance-cover sets. Congressus Numeratium, 162:139–146, 2003.

\bibitem{Sha2006}
K. H. Shafique and R. D. Dutton. A tight bound on the cardinalities of
maximum alliance-free and minimum alliance-cover sets. JCMCC. The
Journal of Combinatorial Mathematics and Combinatorial Computing,
56:139–145, 2006.

\bibitem{Sig2009}
J. M. Sigarreta and J. A. Rodríguez-Velázquez. On the global offensive
alliance number of a graph. Discrete Applied Mathematics, 157(2):219--226, 2009.

\bibitem{Sig2011}
J. M. Sigarreta, I. G. Yero, S. Bermudo, and J. A. Rodríguez-Velázquez.
Partitioning a graph into offensive k-alliances. Discrete Applied Mathematics, 159(4):224–231, 2011.

\bibitem{Yer2011}
I. G. Yero, S. Bermudo, J. A. Rodríguez-Velázquez, and J. M. Sigarreta. Partitioning a graph into defensive k-alliances. Acta Mathematica
Sinica, 27(1):73–82, 2011.

\bibitem{Yer2010}
I. G. Yero and J. A. Rodríguez-Velázquez. Boundary defensive k-
alliances in graphs. Discrete Applied Mathematics, 158:1205–1211,
2010.

\bibitem{Yer2012}
I. G. Yero and J. A. Rodríguez-Velázquez. Partitioning a graph into
global powerful k-alliances. Graphs and Combinatorics, 28:575–583,
2012.

\bibitem{Yer2013c}
Ismael G. Yero and Juan A. Rodríguez-Velázquez. Defensive alliances
in graphs: a survey.
arXiv:1308.2096 [math.CO], 2013.

\bibitem{Yer2013}
I. G. Yero and J. A. Rodríguez-Velázquez. Computing global offensive
alliances in cartesian product graphs. Discrete Applied Mathematics,
161:284–293, 2013.

\end{thebibliography}
\end{document}